\documentclass{scrartcl}
\usepackage{typearea}
\typearea{15}
\usepackage[all]{xy}

\usepackage{comment}

\newcommand{\Class}{\operatorname{Cl}}

\newcommand{\tors}{\operatorname{tors}}
\newcommand{\Bl}{\operatorname{Bl}}
\newcommand{\grmod}{\operatorname{grmod}}

\newcommand{\Kdim}{\operatorname{Kdim}}
\newcommand{\Fun}{\operatorname{Fun}}

\newcommand{\CC}{\mathop{CC}\nolimits}
\newcommand{\dbt}{[\![t]\!]}
\newcommand{\Stab}{\operatorname{Stab}}
\newcommand{\cMquad}{\operatorname{\cM_{\mathrm{quad}}}}

\newcommand{\Mbarquad}{\operatorname{\Mbar_{\mathrm{quad}}}}

\newcommand{\Mbartri}{\operatorname{\Mbar_{\mathrm{tri}}}}
\newcommand{\Mrel}{\operatorname{M_{\mathrm{rel}}}}
\newcommand{\Mbarrel}{\operatorname{\Mbar_{\mathrm{rel}}}}
\newcommand{\Mbarst}{\operatorname{\Mbar_{1,2}}}
\newcommand{\Mbarstb}{\operatorname{\Mbar_{1,1}}}
\newcommand{\Deltarel}{\Delta_{\mathrm{rel}}}
\renewcommand{\ss}{\mathrm{ss}}
\newcommand{\s}{\mathrm{s}}
\newcommand{\irr}{\mathrm{irr}}
\newcommand{\I}{\mathrm{I}}
\newcommand{\Ver}{\mathrm{Ver}}
\newcommand{\GIT}{/\!\!/}
\newcommand{\pibar}{\overline{\pi}}

\usepackage{mymacros}

\newcommand{\bk}{\bsk}

\newcommand{\Isom}{\operatorname{Isom}}
\newcommand{\op}{\mathrm{op}}
\newcommand{\Der}{\operatorname{Der}}
\newcommand{\prim}{\mathrm{prim}}
\newcommand{\nonprim}{\mathrm{non\text{-}prim}}

\newcommand{\bdim}{\mathop{\bfd \bfi \bfm}\nolimits}
\newcommand{\bt}{\bst}
\newcommand{\fm}{\frakm}
\newcommand{\Alg}{\operatorname{Alg}}
\newcommand{\GrAlg}{\operatorname{GrAlg}}

\newcommand{\xc}[1]{\textcolor{red}{#1}}

\title{Compact moduli of noncommutative projective planes}
\author{Tarig Abdelgadir, Shinnosuke Okawa, Kazushi Ueda}
\date{}
\begin{document}

\maketitle

\begin{abstract}
We introduce the notion of the moduli stack of relations of a quiver.
When the quiver with relations is derived-equivalent
to an algebraic variety,
the corresponding compact moduli scheme can be viewed
as a compact moduli of noncommutative deformations of the variety.
We study the case of noncommutative projective planes in detail,
and discuss its relation with geometry of elliptic curves
with level 3 structures.
\end{abstract}


\section{Introduction}

Let $X$ be a smooth projective variety defined over a field $\bk$.
We assume that $\bk$ is an algebraically closed field
of characteristic different from 3
throughout this paper.
The deformation of $X$ as an algebraic variety is controlled
by the first cohomology $H^1(\Theta_X)$ of the tangent bundle,
whereas the deformation of the abelian category $\Qcoh X$ is controlled
by the the second Hochschild cohomology \cite{MR2238922,MR2183254},
which is decomposed as
\begin{align} \label{eq:HKR}
 \HH^2(X) \cong H^2(\cO_X) \oplus H^1(\Theta_X) \oplus H^0(\Lambda^2 \Theta_X) 
\end{align}
by the Hochschild-Kostant-Rosenberg isomorphism
\cite{MR0142598,MR1390671,MR1940241}.
The direct summands of \eqref{eq:HKR} correspond to
gerby, classical, and noncommutative deformations respectively.

Assume that the bounded derived category of coherent sheaves
on the variety $X$ has a full strong exceptional collection
$(E_1, \ldots, E_n)$.
Morita theory for derived categories
\cite{Bondal_RAACS,Rickard} gives an equivalence
\begin{align} \label{eq:Morita}
\RHom_{X} ( \oplus_{i=1}^{n} E_i , - )
\colon
D^b \coh X \simto D^b \module A
\end{align}
with the derived category of finitely-generated right modules
over the total morphism algebra
\begin{align} \label{eq:tma}
 A = \bigoplus_{i,j=1}^n \Hom(E_i, E_j).
\end{align}
The derived equivalence \eqref{eq:Morita} induces
a homotopy equivalence of Hochschild complexes
\begin{align} \label{eq:HH_XA}
 \CC_{\bk}^\bullet(X) \cong \CC_{\bk}^\bullet(A)
\end{align}
as $L_{\infty}$-algebras
\cite{Keller_DI,MR2183254},
which control the deformations of $X$ and $A$ respectively.
This motivates us to study noncommutative deformations of $X$
through deformations of $A$.

The total morphism algebra $A$
is an algebra over the semisimple ring $\bk^{ n}$, and
can be described by a quiver with relations.
Deformations of relations lead to deformations of algebras
over
$
\bk^{n}
$,
and any sufficiently small such deformation of algebras come from
deformation of relations.
Since there is a natural quasi-isomorphism
\begin{align}
 \CC_{\bk}^{\bullet}(A) \cong \CC_{\bk^{n}}^{\bullet}(A)
\end{align}
of dglas
(cf.~e.g.~\cite[Remark 3.10]{Seidel_K3}),
deformation theory of
$
A
$
over
$
\bk^{n}
$
is isomorphic to that over
$
\bk
$.

In this paper, we introduce the moduli stack $\scrR(Q)$
of relations of a quiver $Q$,
and the moduli stack $\scrA$ of finite-dimensional algebras.
There is a natural morphism
from $\scrR(Q)$ to $\scrA$ which sends a relation
to the quotient of the path algebra by the relation.

The latter half of the paper is devoted to a detailed study of the
moduli of relations of the Beilinson quiver with three vertices and its relation
to noncommutative projective planes.
A \emph{3-dimensional Sklyanin algebra}
is the unital associative algebra
$
 S(a,b,c)=\bk \la x,y,z \ra / (f_1,f_2,f_3)
$
generated by three elements $x$, $y$, and $z$
with relations
\begin{align} \label{eq:Sklyanin2}
\begin{split}
 f_1 & = a y z + b z y + c x^2, \\
 f_2 & = a z x + b x z + c y^2, \\
 f_3 & = a x y + b y x + c z^2,
\end{split}
\end{align}
where $(a:b:c)$ is a parameter
in the complement of a finite set in $\bP^2$.
They are originally introduced in \cite{MR684124}
from the point of view of quantum inverse scattering method
and Yang-Baxter equations.
They subsequently attracted much attention
as an important class of
\emph{quadratic AS-regular algebras of dimension 3}
\cite{Artin_Schelter}.

A \emph{noncommutative projective plane}
is an abelian category of the form $\Qgr S$
for a quadratic AS-regular algebra $S$ of dimension 3.
Quadratic AS-regular algebras of dimension 3 are classified
in \cite{Artin_Tate_Van_den_Bergh}
in terms of triples $(E, \sigma, L)$
of a cubic divisor $E$ in $\bP^2$,
an automorphism $\sigma$ of $E$,
and a line bundle $L = \cO_{\bP^2}(1)|_E$ on $E$.

The correspondence
$
(E, \sigma, L)
\mapsto
\Qgr S(E, \sigma, L)
$
from the isomorphism classes of triples
to the equivalence classes of abelian categories
is generically nine to one.
To have a one-to-one correspondence,
one needs to pass from AS-regular algebras
to \emph{AS-regular $\bZ$-algebras}.
Quadratic AS-regular $\bZ$-algebras of dimension 3
are classified in \cite{MR1230966,MR2836401}
by triples $(E, L_0, L_1)$
of a cubic divisor $E$ in $\bP^2$
and two line bundles $L_0$, $L_1$
of degree 3.
These line bundles embed $E$
as a complete intersection
of three bidegree $(1,1)$-hypersurfaces
in $\bP^2 \times \bP^2$,
which is a graph of an automorphism $\sigma$ of $E$.

Classification of triples $(E, L_0, L_1)$ is equivalent
to classification of quadruples $(V_0, V_1, V_2, W)$
consisting of three vector spaces $V_0$, $V_1$, $V_2$
of dimension 3 and a 1-dimensional subspace $W$
of $V_0 \otimes  V_1 \otimes V_2$.
A generator of this subspace is a \emph{potential}
for the relation \eqref{eq:Sklyanin2}.

A noncommutative projective plane has
a full strong exceptional collection,
which is a generalization of the Beilinson collection
\cite{Beilinson} in the commutative case.
The total morphism algebra of this collection
is described by the \emph{Beilinson quiver} $Q$
in \pref{fg:P2_quiver},
equipped with relations $I \subset \bk Q$
coming from \eqref{eq:Sklyanin2}.

\begin{figure}
\begin{align*}
\begin{psmatrix}[rowsep=35mm,colsep=35mm,mnode=circle]
 0 & 1 & 2
\psset{nodesep=3pt,arrows=->}
\nccurve[angleA=15,angleB=165,offset=3pt]{1,1}{1,2}
 \lput*{N}(0.5){x_0}
\ncline[offset=0pt]{1,1}{1,2}
 \lput*{N}(0.5){y_0}
\nccurve[angleA=-15,angleB=195,offset=-3pt]{1,1}{1,2}
 \lput*{N}(0.5){z_0}
\nccurve[angleA=15,angleB=165,offset=3pt]{1,2}{1,3}
 \lput*{N}(0.5){x_1}
\ncline[offset=0pt]{1,2}{1,3}
 \lput*{N}(0.5){y_1}
\nccurve[angleA=-15,angleB=195,offset=-3pt]{1,2}{1,3}
 \lput*{N}(0.5){z_1}
\end{psmatrix}
\end{align*}
\caption{The Beilinson quiver for $\bP^2$}
\label{fg:P2_quiver}
\end{figure}

In this paper,
we give compact moduli schemes
$\Mbartri$, $\Mbarquad$, and $\Mbarrel$ of
triples $(E, L_0, L_1)$,
quadruples $(V_0, V_1, V_2, W)$, and
relations $I \subset \bk Q$.
The \emph{compact moduli of triples}
is defined as the quotient
\begin{align}
 \Mbartri = S(3)/G_{216}
\end{align}
of the elliptic modular surface
of level 3 \cite{MR0429918,MR784140}
by the Hessian group.
The \emph{elliptic modular surface} $S(3)$ is
a compactification of the quotient
$
 S'(3) = (\bH \times \bC) / (\Gamma(3) \ltimes \bZ^2)
$
of $\bH \times \bC$
by the natural action
\begin{align}
 (\gamma,m,n) \colon (\tau, z) \mapsto
  \lb \frac{a\tau+b}{c\tau+d}, \frac{z+m\tau+n}{c\tau+d} \rb
\end{align}
of the semidirect product of
$\Gamma(3) = \Ker(\SL_2(\bZ) \twoheadrightarrow \SL_2(\bF_3))$
acting on $\bZ^2$,
and the \emph{Hessian group}
$G_{216} = \SL_2(\bF_3) \ltimes (\bZ/3\bZ)^2$
acts naturally on it.
The \emph{compact moduli of quadruples}
is the geometric invariant theory quotient
\begin{align}
 \Mbarquad = \bP(V_0 \otimes V_1 \otimes V_2)^\ss
  /\!/ \SL(V_0) \times \SL(V_1) \times \SL(V_2)
\end{align}
of the 26-dimensional projective space
by the natural action of the 24-dimensional group.
The \emph{compact moduli of relations}
is the geometric invariant theory quotient
\begin{align}
 \Mbarrel = \Gr_3(V_0 \otimes V_1)^\ss /\!/ \SL(V_0) \times \SL(V_1)
\end{align}
of the Grassmannian of 3-planes in a 9-dimensional space
by the natural action of the 16-dimensional group.
The main result is the relation among these moduli schemes:

\begin{theorem} \label{th:main}
\ \\[-6mm]
\begin{enumerate}
 \item
$\Mbarquad$ is isomorphic
to the weighted projective plane $\bP(6,9,12)$.
 \item
$\Mbarquad$ and $\Mbarrel$ are naturally isomorphic.
 \item
There is a natural birational morphism $\Mbartri \to \Mbarquad$
which contracts a non-singular rational curve to a point.
 \item
 There exists a natural isomorphism
 $
 \Mbartri
 \cong
 \Mbarst.
 $
 Under this isomorphism,
 the contraction of
 3
 above is identified with the weighted blow-up
 $
 \Mbarst
 \to
 \bP(1,2,3)
 $
 which contracts a boundary prime divisor
 $
 \Delta_{0,2}
 $
 obtained as the section of
 $
 \Mbarst
 \to
 \Mbarstb
 $.
\end{enumerate}
\end{theorem}

Here, $\Mbar_{g,n}$ is the coarse moduli space
of the moduli stack $\cMbar_{g,n}$ of stable curves
of genus $g$ with $n$ marks points.
The moduli space $\Mbarquad$ also appears
as the moduli space of 3-qutrit states
\cite{1402.3768,MR2105225}.
The relevant invariant theory is worked out
in \cite{MR0000221,MR0430168}, and the classification of orbits
and the analysis of their stability are carried out in \cite{MR1361786,MR1348793}.
It is interesting to note that
the image of the contraction $\Mbartri \to \Mbarquad$
can be described either as the quotient of $\bP^2$
by the natural action of $G_{216}$,
or as the quotient of $\bP(V_0 \otimes V_1 \otimes V_2)$
by the natural action of $\SL(V_0) \times \SL(V_1) \times \SL(V_2)$.
We will see in \pref{sc:proof} that this coincidence is a consequence of the invariant theory
due to Vinberg \cite{MR0430168},
applied to a certain graded Lie algebra.

This paper is organized as follows:
In \pref{sc:quiver},
we recall basic definitions on quivers and their relations.
The moduli stack $\scrR(Q)$ of relations on a quiver $Q$
is defined in \pref{sc:moduli_r},
and the moduli stack $\scrA$ of algebras is defined in \pref{sc:moduli_a}.
We discuss the relation between $\scrR(Q)$ and $\scrA$
in \pref{sc:morphism}.
In \pref{sc:mod_rep}, we see the embedding of the moduli of representations of
a quiver with relations into that without relations.
In the case of the Beilinson quiver with three vertices,
this recovers the correspondence between noncommutative projective planes and
elliptic triples.
In \pref{sc:AS-regular},
we recall basic definitions and results
on 3-dimensional quadratic AS-regular algebras
following \cite{Artin_Tate_Van_den_Bergh,
MR2303228,MR2708379}.
In \pref{sc:ASZ},
we collect basic definitions and results
on quadratic AS-regular $\bZ$-algebra
of dimension 3
from \cite{MR1230966,MR2836401}.
The compact moduli scheme of relations of the Beilinson quiver
is studied in \pref{sc:Mrel},
and the compact moduli scheme of triples is studied
in \pref{sc:Mtri}.
The proof of \pref{th:main} is given in \pref{sc:proof} and \pref{sc:isom_with_Mbar_1,2}.

\subsection*{Acknowledgements}
The authors thank Michel Van den Bergh for
communicating the proof of \pref{th:SVdB}.
They also thank Izuru Mori for many important discussions and comments.
We owe a lot to his unpublished manuscript
\cite{Mori_in_preparation}.
T.~A. is supported by Osaka University Invitation Program for Research Abroad.
S.~O. is supported by JSPS Grant-in-Aid for Young Scientists No.~25800017, and
a part of this work was done during his stay at the Max-Planck-Institute
f\"ur Mathematik.
K.~U. is supported by JSPS Grant-in-Aid for Young Scientists No.~24740043.
A part of this work is done
while the authors are visiting Korea Institute for Advanced Study,
whose hospitality and nice working environment
is gratefully acknowledged.

\section{Quivers with relations}
 \label{sc:quiver}

Let us start by recalling the definition of a quiver with relations:

\begin{definition}\ \\[-5mm]
\begin{enumerate}
\item
A {\em quiver} $(Q_0, Q_1, s, t)$ consists of
\begin{itemize}
 \item
a set $Q_0$ of vertices,
 \item
a set $Q_1$ of arrows, and
 \item
two maps
$
s, t \colon Q_1 \rightrightarrows Q_0$ from $Q_1$ to $Q_0$.
\end{itemize}
For an arrow $a \in Q_1$,
The vertices $s(a)$ and $t(a)$
are called the {\em source}
and the {\em target} of $a$ respectively.
 \item
A quiver $(Q_0, Q_1, s, t)$ is \emph{finite}
if both $Q_0$ and $Q_1$ are finite sets.
We will always assume that a quiver is finite in this paper.
 \item
A {\em path} on a quiver
is an ordered set of arrows
$(a_n, a_{n-1}, \dots, a_{1})$
such that $s(a_{k+1}) = t(a_k)$
for $k=1, \dots, n-1$.
The number of arrows in a path
is called the \emph{length} of the path.
We also allow for a path $e_i$ of length zero,
starting and ending at the same vertex $i \in Q_0$.
\item
The {\em path algebra} $\bk Q$
of a quiver $Q$ is the algebra
spanned by the set of paths
as a vector space,
and the multiplication is defined
by the concatenation of paths;
\begin{align}
 (b_m, \dots, b_1) \cdot (a_n, \dots, a_1)
  = \begin{cases}
     (b_m, \dots, b_1, a_n, \dots, a_1) & s(b_1) = t(a_n), \\
      0 & \text{otherwise}.
    \end{cases}
\end{align}
The path algebra is graded by the length of the path.
 \item
An \emph{oriented cycle}
is a path of positive length starting and ending at the same vertex.
The path algebra of a finite quiver is finite-dimensional
if and only if it has no oriented cycles.
 \item
A \emph{quiver with relations} is a pair
$\Gamma = (Q, I)$ of a quiver $Q$
and a two-sided ideal $I$
of the path algebra $\bk Q$.
We will always assume that $I$ is contained
in the two-sided ideal $(\bk Q)_{\ge 2}$
generated by paths of length two.
The \emph{path algebra} of $\Gamma$ is defined as
$\bk \Gamma = \bk Q / I$.
 \item
A \emph{representation} of a quiver is a right module
over the path algebra $A$. 
If
$
V
$
is finite dimensional as a
$
\bk
$-vector space,
the \emph{dimension vector} of an $A$-module $V$
is defined by $(\dim_\bk V e_i)_{i \in Q_0} \in \bZ^{Q_0}$.
The \emph{dimension matrix} of an $A$-bimodule
(i.e., a right $A^{\op} \otimes_\bk A$-module)
is defined similarly as
$(\dim_\bk e_i V e_j)_{i,j \in Q_0} \in \End(\bZ^{Q_0})$.
\end{enumerate}
\end{definition}

The path algebra $A = \bk \Gamma$
of a quiver $\Gamma = (Q, \cI)$ with relations
admits a direct sum decomposition
\begin{align} \label{eq:idemp1}
 A = \bigoplus_{i,j \in Q_0} A_{ij}, \quad
 A_{ij} = e_i A e_j
\end{align}
satisfying
\begin{gather}
 A_{ij} \cdot A_{kl} = 0 \quad \text{if } j \ne k, \\
 A_{ij} \cdot A_{jk} \subset A_{ik}.
\end{gather}
In other words,
it is an algebra over the semisimple algebra $\bk^{Q_0}$
generated by $\{ e_i \}_{i \in Q_0}$.
Note that giving an associative algebra $A$ over $\bk^{Q_0}$
is equivalent to giving an associative algebra $A$
with idempotents $\{ e_i \}_{i \in Q_0}$
satisfying
$
 \sum_{i \in Q_0} e_i = 1.
$
These idempotents give the direct sum decomposition
\eqref{eq:idemp1}.
An element $a \in A_{ij}$ is called \emph{primitive}
if it is not contained in the space of non-primitive paths,
which is defined by
\begin{equation}
A_{ij}^\nonprim
=
\Image{
\left(
\bigoplus_{m \in Q_0} A_{im} \times A_{mj} \to A_{ij}
\right)
}.
\end{equation}


Any finite-dimensional associative algebra $A$
can be described as the path algebra
of a finite quiver with relations:

\begin{proposition} \label{pr:alg_quiver}
Let $Q_0$ be a finite set and
$A$ a finite-dimensional algebra
over $\bk^{Q_0}$.
Then there is a finite quiver $Q$ with the set $Q_0$ of vertices
and a two-sided ideal $I \subset (\bk Q)_{\ge 2}$
such that one has an isomorphism $\bk Q/I \simto A$
of $\bk^{Q_0}$-algebras.
\end{proposition}

\begin{proof}

For each pair of vertices
$
i, j \in Q_0
$,
take a finite set of ``primitive elements''
$
\{ f_{ijk} \}_{k=1}^{d_{ij}}
\subset
A_{ij}
$
so that the set of their classes $\{ [f_{ijk}] \}_{k=1}^{d_{ij}}$ in the vector space 
$
 A_{ij}^\prim := A_{ij}/ A_{ij}^\nonprim
$
is a basis.
Then the set 
$
\{
f_{ijk} | i,j \in Q_0, \, k \in \{ 1, \ldots, d_{ij} \}
\}
$
generates $A$ as a $\bk^{Q_0}$-algebra.
Let $Q$ be the quiver
whose set of vertices is given by $Q_0$
and whose set of arrows is
$Q_1 = \{ a_{ijk} \}_{i,j \in Q_0, \, k \in \{ 1, \ldots, d_{ij} \} }$
such that $s(a_{ijk}) = j$ and $t(a_{ijk}) = i$.
Then one gets a natural epimorphism of $\bk^{Q_0}$-algebras
from $\bk Q$ to $A$ sending $a_{ijk}$ to $f_{ijk}$.
The ideal of relations $I$ is defined as the kernel of this morphism.
\end{proof}

A typical example of an algebra
satisfying the assumption of \pref{pr:alg_quiver} comes
from an exceptional collection.
Let $\scrT$ be a triangulated category
over a field $\bk$.
We will always assume that a triangulated category
has a dg enhancement
in the sense of Bondal and Kapranov
\cite{Bondal-Kapranov_ETC}.
\begin{definition}
An object $E$ of a triangulated category $\scrT$ is \emph{exceptional}
if
\begin{align}
 \Hom^i(E, E) =
\begin{cases}
 \bk \cdot \id_E & i = 0, \\
 0 & \text{otherwise}.
\end{cases}
\end{align}
A sequence $(E_1, \ldots, E_n)$ of objects of $\scrT$
is an \emph{exceptional collection} if
\begin{align}
 \Hom^i(E_j, E_k) = 0
  \text{ for any } i \in \bZ
  \text{ and any } 1 \le k < j \le n.
\end{align}
An exceptional collection $(E_1, \ldots, E_n)$ is \emph{strong} if
\begin{align}
 \Hom^i(E_j, E_k) = 0
  \text{ for any } i \ne 0
  \text{ and any } j, k \in \{ 1, \ldots, n \}.
\end{align}
An exceptional collection $(E_1, \ldots, E_n)$ is \emph{full}
if $\scrT$ is equivalent to its smallest full triangulated subcategory
containing $\{ E_1, \ldots, E_n \}$.
\end{definition}

A triangulated category with a full strong exceptional collection
can be described as the derived category of modules:

\begin{theorem}[{\cite{Rickard,Bondal_RAACS}}]
Let $\scrT$ be a triangulated category with a dg enhancement, and
$(E_1, \ldots, E_n)$ be a full strong exceptional collection in $\scrT$.
Then one has an equivalence
\begin{align}
 \RHom_{\scrT} \lb \bigoplus_{i=1}^{n} E_i , - \rb
 \colon
  \scrT \simto D^b \module A
\end{align}
of triangulated categories
between $\scrT$ and the derived category of finitely-generated right modules
over the total morphism algebra
$
 A = \bigoplus_{i,j=1}^n \Hom(E_i, E_j).
$
\end{theorem}

\section{Moduli of relations}
 \label{sc:moduli_r}

Let $Q = (Q_0, Q_1, s, t)$ be a quiver
without oriented cycles.
For a pair $(i,j)$ of distinct vertices of the quiver $Q$,
we set 
\begin{align}
V_{ij}^\prim=\bigoplus_{
\substack{a \in Q_1
\\
s(a)=j
\\
t(a)=i}}\bk a
\end{align}
and
\begin{align}
 V_{ij}^\nonprim=
 \Image{\lb \bigoplus_{\ell \in Q_0 \setminus \{ i , j \}} e_i (\bk Q) e_\ell
  \otimes e_\ell (\bk Q) e_j\to e_i (\bk Q) e_j \rb}
  =
  e_i \left( ( \bk Q )_{\ge 2} \right) e_j,
\end{align}
so that
$
 e_i (\bk Q) e_j=V_{ij}^\prim \oplus V_{ij}^\nonprim
$
holds.

\begin{proposition} \label{pr:AutQ}
The automorphism group
$
\Aut ( \bk Q / \bk^{Q_0} )
$
is naturally isomorphic to
\begin{align} \label{eq:auto}
  \lb \prod_{i,j} \GL(V_{ij}^\prim) \rb \ltimes
   \lb \prod_{i,j}\Hom_\bk(V_{ij}^\prim, V_{ij}^\nonprim) \rb.
\end{align}
\end{proposition}

\begin{remark}
Note that the group structure on the second term of
\eqref{eq:auto}
is different from the addition of linear maps;
in fact, since we are discussing automorphisms of the path algebra,
we have to rather think of the composition of the maps.
See also \cite{Ye:2011uo} for the automorphism group
of the path algebra.
\end{remark}

\begin{proof}
Since the path algebra $\bk Q$ is generated by arrows,
an element
$
 g \in \Aut \lb \bk Q \relmiddle/ \bk^{Q_0} \rb
$
is determined by its action on arrows.
For an arrow $a \in A_{ij}$,
the element $g(a) \in A$ must lie in $A_{ij}$ again
since $g$ is an automorphism as a $\bk^{Q_0}$-algebra.
This gives an element of
\begin{align}\label{eq:endomorphism}
 \Hom_\bk(V_{ij}^\prim, V_{ij})
  \cong \End_\bk(V_{ij}^\prim)
   \times \Hom_\bk(V_{ij}^\prim, V_{ij}^\nonprim)
\end{align}
for each $i, j \in Q_0$.
Conversely, choosing an element
$
g(a) \in A_{ij}
$
for each
$
a \in A_{ij}
$,
we always get an endomorphism of
$
A
$
as a
$
\bk^{Q_0}
$-algebra.

Consider the representation matrix of
$
g
$
as a linear transformation of
$
\bk Q
$, and its block decomposition corresponding to the direct sum
decomposition of
$
\bk Q
$
by the length of path.
Since
$
g
$
does not decrease the length of path,
we get an upper triangular matrix and hence is
invertible precisely when all the diagonal blocks are
invertible.
Hence one sees that the automorphism group coincides with \eqref{eq:auto},
as a subset of the direct product of \pref{eq:endomorphism}
over all
$
i , j
$.
In order to see that the second term of
\eqref{eq:auto} is a normal subgroup,
note that it is the kernel of the map
\begin{align}
\Aut ( \bk Q / \bk^{Q_0} ) \to \GL ( ( \bk Q )_{ \ge 1 } / ( \bk Q )_{ \ge 2 } ).
\end{align}
\end{proof}

The correspondence between relations of a quiver and
algebras obtained as the quotient of the path algebra
is summarized in the following proposition.

\begin{proposition} \label{pr:AutI}
Let $I, J$ be relations of a quiver $Q$. 
\begin{enumerate}[(1)]
 \item \label{it:AutI1}
$\bk Q/I$ and $\bk Q/J$ are isomorphic as $\bk^{Q_0}$-algebras
if and only if there exists an element $g$ of
$
 \Aut \lb \bk Q \big/ \bk^{Q_0} \rb
$
such that $g(I)=J$.
 \item \label{it:AutI2}
There exists a short exact sequence
\begin{align*}
1
 \to \prod_{i,j\in Q_0} \Hom(V_{ij}^\prim, e_iIe_j)
 \to \Stab_I \lb \Aut \lb \bk Q \big/ \bk^{Q_0} \rb \rb
 \to \Aut \lb (\bk Q/I) \big/ \bk^{Q_0} \rb
 \to 1
\end{align*}
of groups.
\end{enumerate}
\end{proposition}

\begin{proof}
Let $f \colon \bk Q/I \to \bk Q/J$ be an isomorphism
of $\bk^{Q_0}$-algebras.
Since the arrows $ a \in Q_0 $ generate $ \bk Q/I $
as $ \bk^{Q_0} $-algebras,
the homomorphism $f$ is determined by $f(a)$. 
One can lift $f$ to a homomorphism
$g\colon \bk Q \to \bk Q$
by choosing a lift $g(a) \in \bk Q$ of $f(a) \in \bk Q/J$
for each $a \in Q_0$,
since the path algebra $\bk Q$ is freely generated
as a $\bk^{Q_0}$-algebra by the arrows.
The morphism $g$ is surjective, and hence an isomorphism,
since it is surjective on the graded quotient of $\bk Q$
with respect to the length of paths.
One has $g(I) \subset J$
since $g$ is a lift of $f : \bk Q/I \to \bk Q/J$.
This implies $g(I) = J$, and (\ref{it:AutI1}) is proved.

The exact sequence in (\ref{it:AutI2}) follows
from the fact that an element
$g \in \Aut \lb \bk Q \big/ \bk^{Q_0} \rb$
satisfying $g(I) = I$
induces the identity map on $\bk Q/I$
if and only if
$
g(a) - a
\in
I
$
holds for any arrow $a \in Q_1$.
\end{proof}

\begin{definition}
A \emph{sheaf of path algebras} of a quiver $Q$
on a scheme $U$ is a locally-free sheaf $\cQ$ of associative
$\cO_U^{Q_0}$-algebras such that there exists an open covering
$
U = \bigcup_{i} U_i
$
and isomorphisms
\begin{equation}
\varphi_i \colon \cQ|_{U_i} \simto \cO_{U_i} Q
\end{equation}
of $\cO_{U_i}^{Q_0}$-algebras.
\end{definition}

It follows from the definition that
a sheaf of path algebras on $U$ is obtained by
taking an open covering $U = \bigcup_\lambda V_\lambda$
of $U$ and
gluing $\cO_{V_\lambda} Q$ on $V_\lambda$
and $\cO_{V_\mu} Q$ on $V_\mu$
by elements of
$
 \Aut \lb \cO_{V_\lambda \cap V_\nu} Q
  \big/ \cO_{V_\lambda \cap V_\nu}^{Q_0} \rb
$
along $V_\lambda \cap V_\nu$.
Set
$
 G = \Aut \lb \bk Q \big/ \bk^{Q_0} \rb.
$
Then giving a sheaf of path algebras on $U$ is equivalent
to giving a principal $G$-bundle on $U$,
i.e., a morphism $U \to [\Spec \bk/G]$.

The moduli of relations of a quiver is
the category fibered in groupoids
over the category of schemes,
defined as follows:

\begin{definition}
The category $\scrR(Q)$ of relations of a quiver $Q$
is defined as follows:
\begin{itemize}
 \item
An object of $\scrR(Q)$ is a triple $(U, \cQ, \cI)$
consisting of a scheme $U$,
a sheaf $\cQ$ of path algebras of $Q$ on $U$,
and a locally-free subsheaf $\cI$ of two-sided ideals of the sheaf $\cQ$
such that the quotient algebra $\cQ / \cI$
is a flat $\cO_U^{Q_0}$-bimodule.
 \item
A morphism from $(U, \cQ, \cI)$ to $(U', \cQ', \cI')$
is a morphism
$\varphi \colon U \to U'$ of schemes
and an isomorphism $g \colon \cQ \simto \varphi^* \cQ'$
of $\cO_U^{Q_0}$-algebras
such that $g(\cI) = \varphi^* \cI'$.
\end{itemize}
\end{definition}

Note that flatness of $\cQ/\cI$ as a $\cO_U^{Q_0}$-bimodule
implies that the dimension matrix
$\bdim \cI_x = \bdim \cQ_x - \bdim \cQ_x/\cI_x$
is a locally constant function of $x \in U$.
\begin{comment}
This follows from the fact that
locally on $U$
and for any $v, w \in Q_0$,
by taking the base change
with respect to the morphism
$\cQ_0 \to \cO_U e_v \oplus \cO_U e_w$,
one obtains a sheaf $(e_v + e_w) \cQ/\cI (e_v + e_w)$
of $\cO_U e_v \oplus \cO_U e_v$-algebras,
which is flat since flatness is preserved by base change.
Since flat $\cO_U e_v \oplus \cO_U e_v$-algebras
are flat as $\cO_U$-algebras,
the $\cO_U$-rank of $(e_v + e_w) \cQ/\cI (e_v + e_w)$
is locally constant.
Since $Q$ does not have an oriented cycle
so that either $e_v \bk Q e_w$ or $e_w \bk Q e_v$ is zero-dimensional,
and $\dim e_v \bk Q e_v = \dim e_w \bk Q e_w =1$.
It follows that $\dim e_v \bk Q e_w$ and $\dim e_w \bk Q e_v$
are locally constant.
\end{comment}
The category $\scrR(Q)$ can be decomposed
into the disconnected sum
by the dimension matrix $\bdim I \in \End(\bZ^{Q_0})$;
\begin{align}
 \scrR(Q) = \coprod_{\bst \in \End(\bZ^{Q_0})} \scrR(Q; \bt).
\end{align}

\begin{proposition} \label{pr:moduli_r}
Let $Q$ be a quiver without oriented cycles and
$\bt \in \End(\bZ^{Q_0})$ be a dimension matrix.
Then the category $\scrR(Q; \bt)$ is an algebraic stack
of finite type.
\end{proposition}

\begin{proof}
Let
$
 \scrI = \scrI(\bt)
$
be the fine moduli scheme of ideals of the path algebra $\bk Q$
of dimension matrix $\bt$.
It is a closed subscheme of
\begin{equation}\label{eq:relations_simple}
 \prod_{i , j \in Q_0} \Gr (t_{ij} , e_i(\bk Q)e_j )
\end{equation}
cut out by the condition that
the corresponding linear subspace of
$
\bk Q
$
is a two-sided ideal.
Here
$
\Gr ( d , V )
$
denotes the Grassmannian of
$
d
$-dimensional subspaces of
$
V
$.
The group
$
 G = \Aut \lb \bk Q \big/ \bk^{Q_0} \rb
$
of automorphisms of the path algebra $\bk Q$
over the semisimple ring $\bk^{Q_0}$
acts naturally on $\scrI$.
We show that the category $\scrR=\scrR(Q,\bt)$ is equivalent
to the quotient stack
$[\scrI/G]$
by constructing functors $\Phi : \scrR \to [\scrI/G]$
and $\Psi : [\scrI/G] \to \scrR$
which are quasi-inverse to each other.

For an object $(S, \cQ, \cI)$ of $\scrR$,
let $\pi \colon \cP \to S$ be the principal $G$-bundle
associated with $\cQ$,
so that $\pi^* \cQ$ is isomorphic
to the trivial sheaf $\cO_\cP Q$ of path algebras.
Then the pull-back $\pi^* \cI$ of the relation $\cI$
gives an ideal of $\pi^* \cQ \cong \cO_\cP Q$,
so that one obtains a morphism $f \colon \cP \to \scrI$.
This morphism is clearly $G$-equivariant,
and we set $\Phi(S,\cQ,\cI) = f$.
The action of $\Phi$ on morphisms is defined naturally.

Conversely, let
$f \colon \cP \to \scrI$
be an object of $[\scrI/G]$
over a scheme $S$,
so that $\cP$ is a principal $G$-bundle on $S$ and
$f$ is a $G$-equivariant morphism.
Then $G$-equivariance of $f$ ensures
that the pull-back of the universal ideal on $\scrI$
descends to an sheaf $\cI$ of ideals
of the sheaf $\cQ$ of path algebras on $S$
associated with $\cP$,
and we set $\Psi(f) = (S,\cQ,\cI)$.
It is clear that the functor $\Psi$ is quasi-inverse
to the functor $\Phi$,
and \pref{pr:moduli_r} is proved.
\end{proof}

In the simplest case,
the space
$
\scrI ( \bt )
$
of ideals is the whole of \eqref{eq:relations_simple}.
The Beilinson quiver in \pref{fg:P2_quiver},
which is the main subject of this paper,
falls in this class.
As a less trivial example,
consider the case when the quiver comes
from the full strong exceptional collection
$
 (\cO, \cO(f), \cO(s), \cO(s+f))
$
of line bundles on the Hirzebruch surface
$\bF_d=\bP_{\bP^1}(\cO_{\bP^1} \oplus \cO_{\bP^1}(d))$.
Here $s$ and $f$ are the negative section
and a fiber, respectively.
The corresponding quiver is shown in \pref{fg:Fd_quiver},
and the relations will be denoted by $I$.
We label the vertices
and the spaces spanned by arrows between them
as in \pref{fg:Fd_quiver2}.

\begin{figure}[t]
\begin{minipage}{.49 \linewidth}
\begin{align*}
\begin{psmatrix}[rowsep=35mm,colsep=35mm]
 \scO(f) & \scO(s+f) \\
 \scO & \scO(s)
\psset{nodesep=4pt,arrows=->}
\ncline[offset=0pt]{1,1}{1,2}
 \lput*{N}(0.5){a_{d+6}}
\nccurve[angleA=130,angleB=-130,offset=3pt]{2,1}{1,1}
 \lput*{N}(0.5){a_{d+2}}
\nccurve[angleA=95,angleB=-95,offset=0pt]{2,1}{1,1}
 \lput*{N}(0.5){a_{d+3}}
\ncline[offset=0pt]{2,1}{2,2}
 \lput*{N}(0.5){a_{d+1}}
\nccurve[angleA=-70,angleB=160,offset=0pt]{1,1}{2,2}
 \lput*{N}(0.5){a_1}
\ncline[offset=0pt]{1,1}{2,2}
 \lput*{N}(0.5){\reflectbox{$\ddots$}}
\nccurve[angleA=-20,angleB=110,offset=0pt]{1,1}{2,2}
 \lput*{N}(0.5){a_d}
\nccurve[angleA=85,angleB=-85,offset=0pt]{2,2}{1,2}
 \lput*{N}(0.5){a_{d+4}}
\nccurve[angleA=50,angleB=-50,offset=-5pt,nodesep=6pt]{2,2}{1,2}
 \lput*{N}(0.5){a_{d+5}}
\end{psmatrix}
\end{align*}
\caption{The quiver for $\bF_d$}
\label{fg:Fd_quiver}
\end{minipage}
\begin{minipage}{.49 \linewidth}
\begin{align*}
\begin{psmatrix}[rowsep=35mm,colsep=35mm,mnode=circle]
 2 & 4 \\
 1 & 3
\end{psmatrix}
\psset{nodesep=3pt,arrows=->}
\ncline{1,1}{1,2}
 \lput*{N}(0.5){X}
\ncline{2,1}{1,1}
 \lput*{N}{V}
\ncline{2,1}{2,2}
 \lput*{N}{U}
\ncline{2,2}{1,2}
 \lput*{N}{W}
\ncline{1,1}{2,2}
 \lput*{N}{Z}
\end{align*}
\caption{Labeling the vertices and arrows}
\label{fg:Fd_quiver2}
\end{minipage}
\end{figure}

One has
$
 e_3 (\bk Q) e_1 = (Z \otimes V) \oplus U
$
and
$
 e_3 (\bk Q/I) e_1 = H^0(\cO(s)),
$
so that the subspace $I_{31} := e_3 I e_1$
of $(Z \otimes V) \oplus U$ has dimension
\begin{align}
 \dim (Z \otimes V) \oplus U - \dim H^0(\cO(s))
 = (2d+1)-(d+2)
 = d -1
\end{align}
and is spanned by
$$
 a_1a_{d+3}-a_2a_{d+2}, \ 
 a_2a_{d+3}-a_3a_{d+2}, \ 
 \dots, \ 
 a_{d-1}a_{d+3}-a_da_{d+2}.
$$
Similarly, the space
$
 I_{42} := e_4 I e_2
  \subset e_4 (\bk Q) e_2
  = (W \otimes Z) \oplus X
$
has dimension $d-1$ and
is spanned by
$$
 a_{d+5}a_1-a_{d+4}a_2, \ 
 a_{d+5}a_2-a_{d+4}a_3, \ 
 \dots, \ 
 a_{d+5}a_{d-1}-a_{d+4}a_d.
$$
Finally, the subspace $e_4 I e_1$ of
\begin{align}
 Y := e_1 (\bk Q) e_4
  = (X \otimes V) \oplus (W \otimes U) \oplus (W \otimes Z \otimes V)
\end{align}
has dimension $3 d$.
It has a $(3d-2)$-dimensional subspace
which is the image of the map
\begin{equation} \label{eq:contribution_from_shorter_paths}
 \ld e_4 (\bk Q) e_3 \otimes e_3 I e_1 \rd
  \oplus \ld e_4 I e_2 \otimes e_2 (\bk Q) e_1 \rd
  \to e_4 (\bk Q) e_1
\end{equation}
from
\begin{align}
 \ld e_4 (\bk Q) e_3 \otimes e_3 I e_1 \rd
  \oplus \ld e_4 I e_2 \otimes e_2 (\bk Q) e_1 \rd
 = (W  \otimes I_{31}) \oplus (I_{42} \otimes V)
\end{align}
to $e_4 (\bk Q) e_1=Y$.
Hence the space of newly added relations
gives a point in the Grassmannian
$\Gr(2,Y/(W\otimes I_{31}+I_{42}\otimes V))$.
This allows us to describe the moduli stack
$\scrR(Q; \bt)$
of relations of the quiver $Q$ in \pref{fg:Fd_quiver}
with the same dimension matrix $\bt$ as $I$ 
as the following quotient stack.

Let $B$ be the locally-closed subscheme of
$\Gr(d-1, (Z\otimes V)\oplus U)\times \Gr(d-1, (W\otimes Z)\oplus X)$
consisting of pairs $(I_{31}, I_{42})$
such that the rank of the map
\eqref{eq:contribution_from_shorter_paths} is $3d-2$.
Let further
$
 \pi
 \colon
 \scrI
 \to
 B
$
be the Grassmannian fibration
whose fiber over a point $(I_{31},I_{42}) \in B$
is $\Gr(2,Y/(W\otimes I_{31}+I_{42}\otimes V))$.
The group
\begin{align*}
\begin{split}
 \Aut \lb \bk Q \big/ \bk^{Q_0} \rb
 = (\GL(U) \times \GL(V) \times \GL(W) \times \GL(X) \times \GL(Z)) \qquad\\
  \ltimes (\Hom(U, Z \otimes V) \times \Hom(X, W \otimes Z)) 
\end{split}
\end{align*}
acts naturally on
$
\scrI
$,
and one has
$
 \scrR(Q; \bt) = \ld \scrI / \Aut \lb \bk Q \big/ \bk^{Q_0} \rb \rd.
$

\section{Moduli of algebras}
 \label{sc:moduli_a}

\begin{definition} \label{df:moduli_a}
The category $\scrA$ of finite-dimensional unital associative algebras
is defined as follows:
\begin{itemize}
 \item
An object of $\scrA$ is a pair $(U, A)$
of a scheme $U$
and a sheaf $\cA$ of associative unital $\cO_U$-algebras,
which is locally free of finite rank as a $\cO_U$-module. 
 \item
A morphism from $(U,\cA)$ to $(U',\cA')$ is a pair
$(\varphi, f)$ of
a morphism $\varphi \colon U \to U'$ of schemes and
an isomorphism
$
\varphi \colon A \xto[]{\cong} f^*B
$
of
$
\cO_U
$-algebras.
\end{itemize}
\end{definition}

The category $\scrA$ can be decomposed as the disconnected sum
\begin{align}
 \scrA = \coprod_{r \in \bN} \scrA^r
\end{align}
of the subcategory $\scrA^r$
consisting of objects $(U, \cA)$
such that $\rank_{\cO_U} \cA = r$.

Let $V$ be a vector space of dimension $r$, and
$W \subset \Hom_\bk(V \otimes V, V) \times V$
be the set of pairs $(m, e)$ satisfying
the associativity condition
\begin{align}
 m \circ (m \times \id) = m \circ (\id \times m)
\end{align}
and the unit condition
\begin{align}
 m \circ (\id \times e) = m \circ (e \times \id) = \id.
\end{align}
The group $G = \GL(V)$ acts on $W$ from right
through the action
\begin{align}
 (m,e) \mapsto \lb g^{-1}\circ m\circ (g\otimes g), g^{-1}e \rb.
\end{align}

\begin{proposition}\label{pr:quotient_stack}
There is an equivalence of categories fibered in groupoids over $(Sch/\bk)$
between $\scrA^r$ and the quotient stack $[W/G]$.
\end{proposition}

\begin{proof}
We construct functors
$
 F\colon [W/G]\to \scrA^r
$
and
$
 H\colon \scrA^r\to [W/G]
$
which are quasi-inverse to each other.
An object of $[W/G]$ is a triple $(U,P,f)$
consisting of
\begin{itemize}
 \item
 a scheme $U$ over $\bk$,
 \item
 a principal $G$-bundle $\pi\colon P\to U$, and
 \item
 a $G$-equivariant morphism
$
 f\colon P\to W.
$
\end{itemize}
On
$
W
$,
there exists the universal multiplication and the universal identity
\begin{equation}
\begin{split}
m \colon
V \otimes V \otimes \cO_W
\to
V \otimes \cO_W
\\
e \colon
\cO_W \to V \otimes \cO_W.
\end{split}
\end{equation}
Pulling back all these data by
$
f
$
and taking its descent to
$
U
$,
we obtain the corresponding family of associative unital algebras on
$
U
$.

Conversely, to define the functor $H$,
take an object of $\scA$;
it is a pair of a scheme $U$ over $\bk$
and an $\scO_U$-algebra $A$.
Consider the principal $G$-bundle
\begin{align}
 \pi\colon P=\Isom(V\otimes \scO_U, A) \to U,
\end{align}
so that we have the trivialization
\begin{equation}
V \otimes \cO_P
\xto[]{\cong}
\pi^* A.
\end{equation}
Thus we obtain a family of algebra structures on
$
V
$,
and hence the classifying morphism
$
f \colon
P \to W
$.
\end{proof}

\section{A morphism from $\scrR$ to $\scrA$}
 \label{sc:morphism}

Given a quiver and a dimension matrix of relations $(Q, \bt)$,
there is a natural functor
$
F \colon \scrR( Q, \bt ) \to \scrA$
sending an object $(S, \cI , \cQ )$ of $\scrR$
to an object $(S, \cQ / \cI)$ of $\scrA$.
The following proposition shows that
this functor is surjective in an infinitesimal neighborhood
of a point in its image.

\begin{proposition}
Let $(R, \fm)$ be a complete local ring with $\bk = R/\fm$ and
$(Q, I)$ be a quiver with relations.
Then for any object $(\Spec R, A_R)$ of $\scrA$
such that $A_\bk := A_R \otimes_R \bk \cong \bk Q/I$,
there exists an object
$(\Spec R, I_R)$ of $\scrR(Q, \bdim I)$ such that
$A_R \cong R Q/I_R$.
\end{proposition}

\begin{proof}
As pointed out in \cite[Remark 3.10]{Seidel_K3},
a deformation $A_R$ of an algebra over $\bk$
is an algebra over $R^{Q_0}$
if the central fiber $A_\bk$ is an algebra over $\bk^{Q_0}$.
The elements $e_i := \bse_i \cdot \bsone \in A_R$ for $i \in Q_0$
are idempotents of $A_R$,
where $\bse_i$ is the $i$-th idempotent of $R^{Q_0}$ and
$\bsone$ is the unit of $A_R$.
The $R$-module $e_i A_R e_j$ for any $i, j \in Q_0$ is free
since it is a direct summand
of the free $R$-module $A_R$.
Choose an $R$-free basis of $e_i A_R e_j$
in such a way that its reduction mod $\fm$ gives a basis
of $(e_i A_\bk e_j)^\prim$
giving the relation $I_\bk$.
Then the morphism
$R Q \to A_R$
defined by this basis is surjective mod $\fm$,
and hence is surjective by Nakayama's lemma.
The kernel $I_R$ of this morphism is free
since it is a kernel of a surjection of free modules.
\end{proof}

\section{Moduli of representations}
 \label{sc:mod_rep}

Given a quiver $Q$ and a dimension vector $d \in \bZ^{Q_0}$,
one can consider the moduli stack
\begin{align} \label{eq:MQd}
 \scrM(Q;d) = \ld
  \prod_{a \in Q_1} \Hom_{\bk}(\bk^{d_{s(a)}}, \bk^{d_{t(a)}})
   \relmiddle/
   \PGL ( Q ; d)
   \rd  
\end{align}
of representations of $Q$
with the dimension vector $d$.
Here
$
\PGL ( Q ; d)
$
is the quotient of the product group
$
\prod_{v \in Q_0} \GL_{d_v}
$
by its small diagonal
$
\bG_m
$.
Given a two-sided ideal $I \subset \bk Q$,
one gets the closed substack
$
\scrM(Q,I;d)
\subset
\scrM ( Q ; d )
$
of representations of the quiver with relations
$
(Q,I)
$.

Let $d \in \bZ^{Q_0}$ be a dimension vector and
$
\theta \in \Hom(\bZ^{Q_0}, \bQ)
$
a stability parameter satisfying
$
\theta ( d ) = 0
$.
A representation $V$ of dimension vector $d$
is
\emph{
$\theta$-semistable
}
in the sense of King \cite{King}
if $\theta(\dim V') \ge 0$
for any subrepresentation $V' \subset V$.
It is \emph{$\theta$-stable}
if strict inequality holds
for any proper subrepresentation.
The open substack of $\scrM$
consisting of stable (resp. semistable)
representations will be denoted by $\scrM_\theta$
(resp. $\scrM_\theta^\ss$).
A stability parameter is \emph{generic}
if semistability implies stability.

Assume that the quiver $Q$ comes from an exceptional collection,
so that the set $Q_0 = \{ 1, 2, \ldots, n \}$
of vertices is totally ordered
and every arrow $a$ satisfies $s(a) < t(a)$.
When the dimension vector is
$\bsone = (1,\ldots,1) \in \bZ^{Q_0}$,
the moduli stack
$\scrM_\theta(Q; \bsone)$
for a generic stability parameter $\theta$
is a toric variety \cite{MR1648634}.
As an example of a generic stability parameter,
one can take $\theta = (-n-1,1,\ldots,1)$.
A representation $V$ is $\theta$-stable for this $\theta$
if and only if $V$ is generated by the subspace
$V_1 = e_1 V $.

\begin{theorem} \label{th:ideal}
The ideal $I \subset \bk Q$ of relations of $Q$
is determined by the substack
$\scrM(Q, I; \bsone)$ of $\scrM(Q; \bsone)$.
\end{theorem}

\begin{proof}
The homogeneous coordinate ring of the moduli stack
$
 \scrM(Q;\bsone)
  = \ld \bA^{Q_1} \relmiddle/ \bGm^{Q_0} \rd
$
is the polynomial ring
$
 S = \bk[x_a]_{a \in Q_1}
$
graded by the abelian group
$\Hom(\bGm^{Q_0}, \bGm) \cong \bZ^{Q_0}$
as
\begin{align}
 \deg x_a = e_{t(a)} - e_{s(a)}.
\end{align}
Recall that the path algebra $\bk Q$ is graded
by the length of paths.
Consider the natural linear map
\begin{align}
 \phi \colon (\bk Q)_{>0} \to S
\end{align}
which sends the path
$
p = a_n a_{n-1} \cdots a_1
$
to the monomial
$
x_{a_1} \cdots x_{a_n}
$.
This is not a ring homomorphism, but satisfies
\begin{align}
 \phi(a) \cdot \phi(b) = \phi(a \cdot b)
  \quad \text{if $a \cdot b \ne 0 \in \bk Q$}.
\end{align}
The map $\phi$ is injective
since $Q$ does not have an oriented cycle.
With an ideal $I \subset \bk Q$ of relations of $Q$,
one can associate an ideal
$J = (\phi(I))$ of $S$
generated by the image of $I$ by $\phi$.
Note that since we have the decomposition
$
I
=
\oplus_{i , j \in Q_0} e_j I e_i
$,
the ideal
$
J
$
is homogeneous with respect to the
$
\bZ^{Q_0}
$-grading.
It is clear that
$J$ is the defining ideal of the substack
$\scrM(Q,I;\bsone)$ of $\scrM(Q;\bsone)$.
The ideal $J$ is determined by $\scrM(Q,I;\bsone)$
since closed substack of
$\ld \bA^{Q_1} \relmiddle/ \bGm^{Q_0} \rd$
are in one-to-one correspondence
with homogeneous ideals of $S$ with respect to the $\bZ^{Q_0}$-grading.
Finally one can check the equality
\begin{equation}
I
=
\phi^{-1}(J)
\end{equation}
to see that the ideal
$
I
$
is determined by
$
J
$.

%
\begin{comment}
\begin{shaded}
The injective linear map $\phi$
allows one to think of $(\bk Q)_{>0}$
as a linear subspace of $S$, and
one has
\begin{equation}
I
=
\phi^{-1} ( J ),
\end{equation}
so that the correspondence
$
I \mapsto J
$
is injective.
\xc{correspondence of the tangent spaces and surjectivity?}
\end{shaded}
\end{comment}
%
\end{proof}

As an example,
consider the Beilinson quiver for $\bP^2$
shown in \pref{fg:P2_quiver}.
The moduli stack $\cM(Q;\bsone)$ is the quotient
$[\bA^6/(\bGm)^3]$,
and the coarse moduli space $M_\theta(Q; \bsone)$
of $\scrM_\theta(Q; \bsone)$
for $\theta = (-2,1,1)$
is $\bP^2 \times \bP^2$.
The ideal $I$ of relations of $Q$
is a three dimensional subspace of $e_3 (\bk Q) e_1$,
so that the moduli space $C=M_\theta(Q,I;\bsone)$ is 
the complete intersection of three hypersurfaces
of bidegree $(1,1)$.
Let $L_0$ and $L_1$ be the restrictions
of the ample line bundles on the first and the second component
of $\bP^2 \times \bP^2$.
Then the triple $(C, L_0, L_1)$ is the elliptic triple
appearing in the classification
of quadratic AS-regular $\bZ$-algebra of dimension 3
in \pref{th:BP} below.
In this example,
we do not need the information in the unstable locus
to recover the relations of the quiver.

\section{Moduli of non-commutative projective planes}
 \label{sc:nc_P2}
 

\subsection{Quadratic AS-regular algebras of dimension 3}
 \label{sc:AS-regular}

An associative unital graded algebra
$S = \bigoplus_{n=0}^\infty S_n$
is \emph{connected} if $S_0 = \bk \cdot 1_S$.
The positive part $S_+ = \bigoplus_{n=1}^\infty S_n$
of a connected algebra $S$ is a two-sided ideal,
and the quotient module $\bk = S/S_+$ is a simple object
in the category $\gr S$ of finitely-generated graded right $S$-modules.
A connected algebra
is a \emph{quadratic AS-regular algebra of dimension 3}
\cite{Artin_Schelter,Artin_Tate_Van_den_Bergh}
if the simple module $\bk$ has a projective resolution
of the following form:
\begin{align} \label{eq:AS-regular}
 0 \to S(-3) \to S(-2)^{\oplus 3} \to S(-1)^{\oplus 3} \to S \to \bk \to 0.
\end{align}
The quotient of the abelian category $\gr S$
by the Serre subcategory $\tor S$
consisting of finite-dimensional modules
will be denoted by
\begin{align}
 \pi : \gr S \to  \qgr S := \gr S / \tor S.
\end{align}
A \emph{non-commutative projective plane}
is a $\bk$-linear abelian category
which is equivalent to $\qgr S$
for some $3$-dimensional quadratic regular
algebra $S$. 
The objects $\pi(S(n))$ of $\qgr S$
will be denoted by $\cO(n)$.
One has the following well-known generalization
of a theorem of Beilinson \cite{Beilinson}:

\begin{theorem} \label{th:Beilinson}
The sequence $(\scO, \scO(1), \scO(2))$
is a full strong exceptional collection
on any noncommutative projective plane.
\end{theorem}

\begin{proof}
Let $\scrT$ be the full triangulated subcategory $\scrT$ of $D^b \qgr S$
containing $\cO$, $\cO(1)$, and $\cO(2)$.
It is an admissible subcategory since $(\cO, \cO(1), \cO(2))$
is an exceptional collection.
Since the functor $\pi$ is exact,
the exact sequence \eqref{eq:AS-regular}
induces the exact sequences
\begin{align}
 0 \to \scO(i) \to \scO(i+1)^{\oplus 3} \to \scO(i+2)^{\oplus 3} \to \scO(i+3) \to 0
\end{align}
for $i \in \bZ$.
It follows that $\scrT$ contains $\cO(n)$ for all $n \in \bZ$.
Note that for any non-zero object $\cM$ of $\qgr S$,
one has $\Hom(\cO(-n), \cM) \ne 0$
for sufficiently large $n$.
This implies that the right orthogonal $\scrT^\bot$ vanishes,
and \pref{th:Beilinson} is proved.
\end{proof}

Quadratic AS-regular algebras of dimension 3 are classified
in \cite{Artin_Tate_Van_den_Bergh}
in terms of triples $(E, \sigma, L)$,
where $E$ is either $\bP^2$ or a cubic divisor in $\bP^2$,
$\sigma$ is an automorphism of $E$ as a scheme, and
$L=\cO_{\bP^2}(1)|_E$ is a line bundle on $E$.
The AS-regular algebra
associated with a triple $(E,\sigma,L)$ is the quotient
\begin{align}
 S(E,\sigma,L)=T S_1 / \la R \ra
\end{align}
of the free tensor algebra $T S_1$ over
\begin{align}
 S_1
 =H^0(E, L)
 \cong H^0(\bP^2, \scO_{\bP^2}(1))
\end{align}
by the two-sided ideal generated by the 3-dimensional subspace
\begin{align}
 R
  &= \Ker \lb S_1\otimes S_1\to H^0(E, L \otimes \sigma^* L) \colon
   s \otimes s' \mapsto s \cdot \sigma^* s' \rb \\
  &\cong H^0( E\times E, L \boxtimes L)(-\Gamma_{\sigma})).
\end{align}
\begin{comment}
\begin{shaded}
This is formerly written as
\begin{align}
 R
  = \ker \lb S_1\otimes S_1\to H^0(E, \scO_E(2)) \rb
  = H^0(E\times E, \scO(1)\boxtimes\scO(1)(-\Gamma_{\sigma})).
\end{align}
Is the new version correct?
\end{shaded}
\end{comment}
Here $\Gamma_\sigma \subset E \times E \subset \bP^2 \times \bP^2$
is the graph of the automorphism $\sigma : E \to E$.
When $E=\bP^2$,
the resulting noncommutative projective plane
$\qgr S(E,\sigma,L)$ is equivalent to the commutative projective plane
$\coh \bP^2$
as an abelian category.
If $E$ is an elliptic curve and $\sigma$ is a translation,
then $S(E,\sigma,L)$ is called a \emph{Sklyanin algebra}.
A Sklyanin algebra is generated by three elements
with three quadratic relations;
\begin{align} \label{eq:Sklyanin3}
 S(a,b,c) &= \bk \la x, y, z \ra / (f_1, f_2, f_3), \\
\begin{split}
 f_1 &= a y z + b z y + c x^2, \\
 f_2 &= a z x + b x z + c y^2, \\
 f_3 &= a x y + b y x + c z^2.
\end{split}
\end{align}
A triple $(a:b:c) \in \bP^2$ fails to give a Sklyanin algebra
if and only if $(a:b:c)$ belongs to
\begin{align}
 \Delta = \lc (a:b:c) \in \bP^2 \relmid a^3=b^3=c^3 \rc
  \cup \{(1:0:0)\} \cup \{(0:1:0)\} \cup \{(0:0:1)\}.
\end{align}
Given a point $(a:b:c)\in\bP^2\setminus\Delta$,
one has
\begin{align}
 S(a,b,c)=S(E,\sigma,L) 
\end{align}
where
$E$ is the Hesse cubic
\begin{align}
 E = \{ (x:y:z) \in \bP^2 \mid
  ( a^3 + b^3 + c^3 ) x y z - a b c ( x^3 + y^3 + z^3 ) = 0 \},
\end{align}
$\sigma : E \to E$ is the translation
by the point $(a:b:c) \in E$
with respect to the origin $(1:-1:0)$,
and $L = \cO_{\bP^2}(1)|_E$.
The graph 
$
 \Gamma(\sigma) \subset
\bP^2_{x_0 y_0 z_0}
\times
\bP^2_{x_1 y_1 z_1}
$
of $\sigma$ is defined by
\begin{align} \label{eq:graph}
\lc
\begin{aligned}
 \ftilde_1 &= a y_0 z_1 + b z_0 y_1 + c x_0 x_1, \\
 \ftilde_2 &= a z_0 x_1 + b x_0 z_1 + c y_0 y_1, \\
 \ftilde_3 &= a x_0 y_1 + b y_0 x_1 + c z_0 z_1.
\end{aligned}
\right.
\end{align}

\subsection{Quadratic AS-regular $\bZ$-algebras of dimension 3}
 \label{sc:ASZ}

To study noncommutative geometry,
it is sometimes useful to work with
$\bZ$-algebras instead of $\bZ$-graded algebras.

\begin{definition}[\cite{MR1230966,MR2836401}]
\ \\[-6mm]
\begin{itemize}
 \item
A \emph{$\bZ$-algebra} is a pre-additive category
whose objects are indexed by $\bZ$.
A $\bZ$-algebra can be viewed as a ring
$
 A=\bigoplus_{i,j\in\bZ}A_{ij}
$
where
$
A_{ij}
=
\Hom( j , i )
$
and the multiplication is defined
by the composition of the category.
The identity elements $e_i \in A_{ii}$ are called
\emph{local units}.
 \item
A \emph{right $A$-module} is a $\bZ$-graded vector space
$M=\bigoplus_{i\in\bZ}M_i$
with a right action of $A$
satisfying $M_i A_{jk} = 0$ for $i \ne j$ and
$M_i A_{ik} \subset M_k$ for any $i,j,k\in\bZ$.
The category of right $A$-modules will be denoted by $\Gr A$.
 \item
A $\bZ$-algebra $A = \bigoplus_{i,j \in \bZ} A_{ij}$ is \emph{connected}
if $A_{ij}=0$ for $i > j$,
$\dim A_{ij}<\infty$ for $i \le j$,
and
$A_{ii} = \bk e_i$.
The projective module $e_i A$ and the simple module
$e_i A e_i$ will be denoted by $P_i$ and $S_i$.
 \item
A \emph{quadratic AS-regular $\bZ$-algebra of dimension 3}
is a connected $\bZ$-algebra such that
the minimal resolution of $S_i$ has the form
\begin{equation}
0\to P_{i+3}\to P_{i+2}^3 \to P_{i+1}^3 \to P_i\to S_i\to 0.
\end{equation}
 \item
A right $A$-module is a \emph{torsion module}
if it is a colimit of objects $M$ satisfying $M_i = 0$
for $i \gg 0$.
The full subcategory of $\Gr A$
consisting of torsion modules
is denoted by $\Tor A$.
It is a Serre subcategory of $\Gr A$,
and the quotient abelian category is denoted
by $\Qgr A = \Gr A / \Tor A$.
 \item
A $\bZ$-algebra $A$ is \emph{Noetherian}
if the category $\Gr A$ is a locally Noetherian Grothendieck category.
The full subcategory of $\Gr A$ and $\Qgr A$
consisting of Noetherian objects is denoted
by $\gr A$ and $\qgr A$ respectively.
\end{itemize}
\end{definition}

One advantage of using $\bZ$-algebras
instead of $\bZ$-graded algebras is the following theorem,
which is proven in Appendix.

\begin{theorem}[{\cite[Theorem 11.2.3 and Corollary 11.2.4]{MR1816070}}]
 \label{th:SVdB}
For any noncommutative projective plane $\qgr S$,
there is a unique quadratic AS-regular $\bZ$-algebra $A$
such that $\qgr S \cong \qgr A$.
\end{theorem}
This allows one to classify noncommutative projective plane
up to equivalence of abelian categories
in terms of quadratic AS-regular $\bZ$-algebras
up to isomorphisms.
This is in contrast with the existence
of non-isomorphic quadratic AS-regular algebras
$S$ and $S'$ such that $\qgr S \cong \qgr S'$.
A quadratic AS-regular $\bZ$-algebra $A$ of dimension 3
is \emph{linear} if $\qgr A \cong \coh \bP^2$,
and \emph{elliptic} otherwise.

\begin{definition} \label{df:admissible}
A triple $(C, L_0, L_1)$ consisting of
a scheme $C$ and two line bundles $L_0$ and $L_1$ on $C$
is \emph{admissible}
if
\begin{enumerate}[(i)]
 \item \label{it:elliptic1}
for both $i=0$ and $i=1$,
the complete linear system associated with $L_i$
embeds $C$ as a divisor of degree $3$ in $\bP^2$,
 \item \label{it:elliptic2}
$\deg(L_0|_D)=\deg(L_1|_D)$ holds
for any irreducible component $D$ of $C$, and
 \item \label{it:elliptic3}
$L_0$ and $L_1$ are not isomorphic.
\end{enumerate}
Additionally, the \emph{linear triple} is defined by
$
 (C, L_0, L_1)
  = (\bP^2, \cO_{\bP^2}(1), \cO_{\bP^2}(1)).
$
Two triples $(C,L_0,L_1)$ and $(C',L_0',L_1')$ are \emph{isomorphic}
if there is an isomorphism $\varphi : C \to C'$
of schemes such that $\varphi^* L_0' \cong L_0$ and
$\varphi^* L_1' \cong L_1$.
\end{definition}

\begin{theorem}[{\cite{MR1230966},
\cite[Proposition 3.3]{MR2836401}}]
 \label{th:BP}
Isomorphism classes
of elliptic quadratic AS-regular $\bZ$-algebras of dimension 3
are in one-to-one correspondence
with isomorphism classes of admissible triples.
\end{theorem}

\begin{definition}[{\cite[p.245, Definition]{MR1230966}}]
A \emph{geometric quadruple} $(V_0, V_1, V_2, W)$ consists
of $3$-dimensional vector spaces $V_0, V_1, V_2$ and
a one-dimensional subspace
$
 W\subset V_0\otimes V_1\otimes V_2,
$
spanned by an element
$
 w \in V_0\otimes V_1\otimes V_2
$
satisfying the following \emph{geometricity} condition;
for any $i = 0,1,2$ and any element $0\not=v^*_i\in V^*_i$,
the contraction
\begin{align}
 w\rfloor v_i^*\in V_{i+1}\otimes V_{i+2}
\end{align}
of $w$ by $v_i$ has rank at least two.
Here the rank of an element
$
\lambda
\in
V_{i+1} \otimes V_{i+2}
$
is defined to be that of the linear map
\begin{equation}
V_{i+1}^* \to V_{i+2} \ ; \ v_{i+1}^* \mapsto \lambda \rfloor v_{i+1}^*,
\end{equation}
and the indices $i$, $i+1$, and $i+2$ are taken modulo $3$.
An \emph{isomorphism} from a geometric quadruple
$(V_0, V_1, V_2, W)$
to another geometric quadruple
$(V'_0, V'_1, V'_2, W')$
is a triple $(\varphi_1, \varphi_2, \varphi_3)$
of isomorphisms
$\varphi_i \colon V_i \to V'_i$
satisfying
$
 (\varphi_0\otimes\varphi_1\otimes\varphi_2)(W)=W'.
$
\end{definition}

\begin{theorem}[{\cite{MR1230966,MR2836401}}] \label{th:trinity_ncp2}
There is a one-to-one correspondence
between the set of isomorphism classes of
\begin{itemize}
 \item
quadratic regular $\bZ$-algebras of dimension 3,
 \item
linear or admissible triples, and
 \item
geometric quadruples.
\end{itemize}
\end{theorem}

The one-to-one correspondence
in \pref{th:trinity_ncp2}
is given as follows:
For a triple $(C, L_0, L_1)$, define the \emph{helix} $(L_n)_{n\in \bN}$ by
\begin{align}
L_n=L_0\otimes (L_1\otimes L_0^{-1})^{\otimes n}.
\end{align}
We put $V_i=H^0(C, L_i)$ and
\begin{align}
R_i= \ker (V_i\otimes V_{i+1}\to H^0(C, L_i\otimes L_{i+1})).
\end{align}
Then the $\bZ$-algebra
$
 A=\bigoplus_{i\le j}A_{ij}
$ associated to the triple $(C, L_0, L_1)$
is defined as the quotient of the free $\bZ$-algebra generated by $V_i=A_{i,i+1}$
by the ideal generated by the relations
$
 R_i\subset V_i\otimes V_{i+1}=A_{i,i+2}.
$

For a quadratic AS-regular $\bZ$-algebra $A$ of dimension 3,
the corresponding geometric quadruple $(V_0, V_1, V_2, W)$
is defined by
\begin{align}
 V_i=A_{i, i+1}
\end{align}
for $i=0,1,2$ and
\begin{align}
 W=(R_0\otimes V_2)\cap (V_0\otimes R_1),
\end{align}
where
\begin{align}
 R_i=\ker (V_i\otimes V_{i+1}\to A_{i, i+2}).
\end{align}
For a geometric quadruple $(V_0, V_1, V_2, W)$,
the corresponding triple is defined as follows.
Take a basis $w$ of $W$ and regard it as a map
\begin{align}
 \varphi\colon V_0^*\to V_1\otimes V_2.
\end{align}
Let $C$ be the closed subscheme of
$
 \bP_{\bullet} V_1\times \bP_{\bullet}V_2\cong \bP^2\times\bP^2
$
cut out by the three independent hypersurfaces of bidegree $(1,1)$
defined by the image of $\varphi$.
Then we define the line bundles $L_0, L_1$ to be the pull-backs of the $\scO_{\bP^2}(1)$ by the
first and the second projections.

For a quadratic AS-regular algebra $S(E, \sigma, L)$
of dimension 3
coming from an elliptic curve $E$,
an automorphism $\sigma : E \to E$,
and a line bundle $L$,
the quadratic AS-regular $\bZ$-algebra of dimension 3
associated with $\qgr S(E,\sigma, L)$
comes from the elliptic triple
$(E, L, \sigma^*L)$
(cf. \cite[Section 3.2]{MR2836401}).

\subsection{Compact moduli of relations of the Beilinson quiver}
 \label{sc:Mrel}

The total morphism algebra
$\bigoplus_{i,j=0}^2 \Hom(\cO(i), \cO(j))$
of the Beilinson collection
$(\cO,\cO(1),\cO(2))$
on a noncommutative projective plane
is described by the Beilinson quiver
in \pref{fg:P2_quiver}
with the dimension matrix
\begin{align}
 t_{ij} =
\begin{cases}
 3 & (i,j)=(3,1), \\
 0 & \text{otherwise}.
\end{cases}
\end{align}
We have the isomorphism
\begin{align}
 \scrR
  = \scrR(Q;\bt)
  \cong [\Gr(3, V_0 \otimes V_1)/\GL(V_0) \times \GL(V_1)]
\end{align}
of stacks,
where $V_i$ are $\bk$-vector spaces of dimension three.

The center of the group $\GL(V_0) \times \GL(V_1)$ acts
trivially on $\Gr(3, V_0 \otimes V_1)$,
so that we have the following isomorphism of stacks;
\begin{align}
 \scrR \cong [\Gr(3, V_0 \otimes V_1)/\PGL(V_0) \times \PGL(V_1)]
  \times [\Spec \bk /\bG_m\times \bG_m].
\end{align}
The corresponding GIT quotient
\begin{align}
 \Mbarrel=\Gr(3, V_0 \otimes V_1)^{\ss} \GIT \SL(V_0) \times \SL(V_1)
\end{align}
will be called the \emph{compact moduli of relations} of the Beilinson quiver.
Since
$
 \Pic \Gr(3, V_0 \otimes V_1) \cong \bZ
$
and
$
 \SL(V_0) \times \SL(V_1)
$
has no nontrivial character, there is no VGIT here.

Recall that the Grassmannian
$
 \Gr(3, V_0 \otimes V_1)
$
has the following description as a GIT quotient;
\begin{align}
 \Gr(3, V_0 \otimes V_1)
 \cong
 \Hom(V_2^{\vee}, V_0 \otimes V_1)^{\s}/\GL(V_2)
 \subset
 [V_0 \otimes V_1 \otimes V_2/ \GL(V_2)].
\end{align}
Here $V_2$ is another $3$-dimensional vector space and
$
 \subset
$
is an open immersion.
Hence there exists the following canonical open immersion
\begin{align}
\begin{split}
 \scrR
  &\cong [\Gr(3, V_0 \otimes V_1)/\GL(V_0) \times \GL(V_1)] \\
  &\subset [V_0 \otimes V_1 \otimes V_2
   / \GL(V_0) \times \GL(V_1) \times \GL(V_2)] \\
  &=: \cMquad
\end{split}
\end{align}
of stacks.
This induces a canonical morphism between
the corresponding GIT quotients, or the compact moduli,
which turns out to be an isomorphism;
\begin{align}
 \Mbarrel
  \cong \bP_{\bullet}(V_0 \otimes V_1 \otimes V_2)^{\ss}
   \GIT \SL(V_0) \times \SL(V_1)\times \SL(V_2)
  =: \Mbarquad.
\end{align}

The orbits and the ring of semi-invariants of this action
have been studied by many people,
including both mathematicians and physicists.
We collect some of the known results below:

\begin{theorem}
\label{th:git_of_ncp2}
Consider the standard action of the group
$
 \SL(V_0) \times \SL(V_1)\times \SL(V_2)
$
on
$
 \bP_{\bullet}(V_0 \otimes V_1 \otimes V_2).
$
\begin{enumerate}[(i)]
 \item \label{it:ng_proposition2}
The stable locus coincides with the non-vanishing locus of the
hyperdeterminant $\Delta$, which is an invariant of degree $36$.
 \item \label{it:ng_theorem1}
The closed points of the stable locus
are in one-to-one correspondence with the isomorphism classes of
admissible triples $(E, L_0, L_1)$ with $E$ non-singular.
 \item \label{it:veronese_cuboid}
The orbit in
$
\bP_{\bullet}(V_0 \otimes V_1 \otimes V_2)
$
corresponding to the two-sided ideal of the commutative
$
\bP^2
$
is strictly semi-stable.
 \item \label{it:the_invariant_subring}
The invariant subring
\begin{align}
 (\Sym^{\bullet}_{\bk}
 (V_0^{\vee}\otimes V_1^{\vee}\otimes V_2^{\vee}))
 ^{\SL(V_0) \times \SL(V_1)\times \SL(V_2)}
\end{align}
is freely generated by three invariants $I_6, I_9, I_{12}$ of degree $6,9,12$, so that
the compact moduli of relations is isomorphic to the weighted projective plane
$\bP(6,9,12)$.
\end{enumerate}
\end{theorem}

The proofs of (\pref{it:ng_proposition2}),
(\pref{it:ng_theorem1}), and
(\pref{it:veronese_cuboid})
can be found in
\cite[Proposition 2, p. 97]{MR1348793},
\cite[Theorem 1, p. 56]{MR1361786}, and
\cite[p. 92]{MR1348793} respectively.
The commutative $\bP^2$ corresponds
to the orbit of the \emph{Veronese cuboid}
(see \cite[p. 92]{MR1348793}),
which forms a polystable orbit.
The proof of (\pref{it:the_invariant_subring}) is first given
by \cite{MR0000221},
and also follows from the invariant theory of Vinberg
\cite{MR0430168}.
See also \cite{MR2105225} and references therein.

\subsection{Compact moduli of triples}
\label{sc:Mtri}

The \emph{Hesse family} of elliptic curves
(cf.~e.g.~\cite{Artebani-Dolgachev} and references therein)
is defined by
\begin{align} \label{eq:Hesse}
 S(3) := \{ ((x:y:z),(t_0:t_1) \in \bP^2_{x:y:z} \times \bP^1_{t_0:t_1}
  \mid t_0(x^3+y^3+z^3)+t_1xyz=0 \}
  \to \bP^1_{t_0:t_1}.
\end{align}
The generic member of this family is a smooth elliptic curve,
which degenerates to a triangle of three lines at the four points
\begin{align}
 (t_0:t_1)=(0:1), (1:-3), (1: -3\omega), (1: -3\omega^2),
\end{align}
where
$
 \omega
$
is a primitive third root of unity.
The Hesse pencil has nine base points,
which form the set of inflection points for any smooth member
$
 E_{t_0 : t_1}
$
of the Hesse family.
We set the origin to be
$
o = (1 : -1 : 0)
$,
so that the nine base points form the set
$
 E_{t_0 : t_1}[3]
$
of three-torsion points.
The group law and the translation by a point $p$
will be denote by $\oplus$, $\ominus$, and $\tau_p(-) = p \oplus (-)$.
By blowing up $\bP^2$ at the base locus,
one obtains a rational elliptic surface
\begin{align} \label{eq:blow-up}
 \pi \colon \Bl_{p_0,\dots, p_8} \bP^2 \to \bP^1_{t_0 : t_1},
\end{align}
which is isomorphic to the Hesse family \eqref{eq:Hesse}.
The exceptional curves
$
 D_0,\dots, D_8
$
are sections of $\pi$,
which give 3-torsion points
in a smooth fiber.

The Hesse family gives a projective model
for the Shioda's elliptic modular surface
$
S(3) \to X(3)
$.
It is a natural compactification of the family
\begin{align}
 S'(3) := (\bH \times \bC)/(\Gamma(3) \ltimes \bZ^2)
  \to X'(3) := \bH/\Gamma(3) ,
\end{align}
where $\bH$ is the upper half plane and
\begin{align}
 \Gamma(3)=\ker \lb \SL_2(\bZ) \to \SL_2(\bZ/3\bZ) \rb
\end{align}
is the principal congruence group of level 3
\cite{MR784140,MR0429918}.
The action of $\Gamma(3) \ltimes \bZ^2$
on $\bH \ltimes \bC$ is given by
\begin{align}
 (\gamma,m,n) \colon (\tau, z) \mapsto
  \lb \frac{a\tau+b}{c\tau+d}, \frac{z+m\tau+n}{c\tau+d} \rb.
\end{align}
The collection
$
 (\pi \colon S'(3) \to X'(3), (D_0 , D_1, D_3))
$
is a universal family of elliptic curves with level-3 structures,
where the section $D_0$ gives the origin
$
o
$
and
the sections $D_1$ and $D_3$ give
a basis of the group of 3-torsion points.

The residual action of the \emph{Hessian group}
\begin{align}
 G_{216} =
  \lb \SL_2(\bZ) \ltimes \lb \frac{1}{3}\bZ \rb^2 \rb /
  (\Gamma(3) \ltimes \bZ^2)
 \cong
 \SL_2(\bZ / 3\bZ) \ltimes (\bZ / 3)^2
\end{align}
on $S'(3)$ extends to the surface $S(3)$
(see \cite[p. 78]{MR784140}).
The Hessian group is identified,
through the blow-up morphism
$S(3)\to \bP^2$,
with the subgroup of $\PGL_3(\bk)$
preserving the Hesse pencil.

To any point $p \in S'(3) \setminus (D_0 \cup \cdots \cup D_8)$,
one can associate an admissible triple
\begin{align}\label{eq:the_triple_associated_to_a_point_on_S'(3)}
 (E_p := \pi^{-1}(\pi(p)),
  L_0 = \cO_{E_p}(3 D_0 \cap E_p),
  L_1 = \cO_{E_p}(3p)),
\end{align}
which forms a family $(\cE, \cL_0, \cL_1)$ of admissible triples
over $S'(3) \setminus (D_0 \cup \cdots \cup D_8)$.

\begin{lemma} \label{lm:Hessian-orbit}
Let $p$ and $q$ be two points
on $S'(3) \setminus (D_0 \cup \cdots \cup D_8)$.
\begin{enumerate}
 \item
The triples
$
(E_p , L_0 , \tau_p)
$
and
$
(E_q , L_0 , \tau_q)
$
are isomorphic if
$
p
=
g(q)
$
for some
$
g \in \SL_2( \bF_3 )
$.
 \item
The triples
$
(E_p , L_0 , L_1)
$
and
$
(E_q , L_0 , L_1)
$
are isomorphic if
$
p
=
g(q)
$
for some
$
g \in G_{216}
$.
\end{enumerate}
\end{lemma}

\begin{proof}
\begin{enumerate}
\item
Note that
$
g
\in
\SL_2( \bF_3 )
$
induces an isomorphism
$
g \colon
E_p \xto[]{\cong} E_{g(p)}
$
which preserves the origins.
Conversely, suppose that there exists
an isomorphism
\begin{equation}
\varphi
\colon
( E_p , L_0 , \tau_p )
\xto[]{\cong}
( E_q , L_0 , \tau_q ).
\end{equation}
Since
$
\varphi
$
respects the line bundles
$
L_0 = \cO (3o)
$,
by composing
$
\varphi
$
with a three-torsion translation of
$
E_q
$
if necessary, we can assume that
$
\varphi
$
respects the origins
$
o
$.
Since
$
\varphi
$
also respects the translations
$
\tau_p
$
and
$
\tau_q
$,
we see
$
\varphi(p)
=
q
$.
By passing to the universal covers,
we can find an element
 $
 g
 \in
 \SL_2 (\bF_3)
 $
acting as a lift of
$
\varphi
$.

\item

It is clear that
$
g
\in
G_{216}
=
 \SL_2(\bZ / 3\bZ) \ltimes (\bZ / 3)^2
$
induces an isomorphism of the triples.
Conversely consider an isomorphism
\begin{equation}
\psi
\colon
( E_p , \cO(3o) , \cO(3p) )
\xto[]{\cong}
( E_q , \cO(3o) , \cO(3q) ).
\end{equation}
Again composing
$
\psi
$
with a three-torsion translation, we can assume
$
\psi
$
respects the origins and find
$
g
\in
\SL_2(\bF_3)
$
as before.
Finally since
$
\varphi
$
sends
$
\cO(3p)
$
to
$
\cO(3q)
$,
one can find
$
h
\in
(\bZ / 3)^2
$
such that
$
( h \circ g )(p)
=
q
$.
Thus we conclude the proof.
\end{enumerate}\end{proof}

\begin{definition}
The \emph{compact moduli scheme $\Mbartri$ of triples} is defined as
$S(3)/G_{216}$.
\end{definition}

The exceptional divisors $D_0,\ldots,D_8$ in $S(3)$
are mapped to the same smooth rational curve $D$
in the quotient $S(3)/G_{216}$,
which is isomorphic to the modular curve $X(1)$
with respect to $\SL_2(\bZ)$.
The open subscheme
$
 (S'(3)/G_{216}) \setminus D
 \subset \Mbartri
$
is the coarse moduli scheme of
non-singular admissible triples.
Since the action of the group $G_{216}$
preserves the union $D_0 \cup \cdots \cup D_8$
of the exceptional divisors,
there exists a birational morphism
$
 S(3)/G_{216} \to \bP^2/G_{216}
$
which contracts the curve $D$ to a point.
The action of $G_{216}$ on $S(3)$ identifies
the three irreducible components of the four singular fibers,
and their image in the quotient
$S(3)/G_{216}$
is the cuspidal rational curve of arithmetic genus one.

%
%

\subsection{From triples to relations}
\label{sc:proof}

For each point of
$
S'(3) \setminus (D_0 \cup \cdots \cup D_8)
$,
we can construct an elliptic triple as in
\pref{eq:the_triple_associated_to_a_point_on_S'(3)}.
From the family of triples thus obtained,
we can construct a family of geometric quadruples over
$
 S'(3) \setminus (D_0 \cup \cdots \cup D_8)
$
by the correspondence of \pref{th:trinity_ncp2}.
Hence we obtain the functorial morphism
\begin{equation}\label{eq:Ftilde}
 \Ftilde \colon S'(3) \setminus (D_0 \cup \cdots \cup D_8) \to \scrR
\end{equation}
to the stack $\scrR$
of relations of the Beilinson quiver.

\pref{th:git_of_ncp2}(\ref{it:ng_theorem1}) shows that
we see that the image of the morphism $\Ftilde$ lands
in the stable locus.
Hence one obtains a morphism
$
 \Ftilde \colon 
 S'(3) \setminus (D_0 \cup \cdots \cup D_8)
 \to
 \Mrel.
$
\pref{lm:Hessian-orbit} shows that
$
\Ftilde
$
factors through the quotient by the group
$
G_{216}
$,
so that we obtain the rational map
%
\begin{equation} \label{eq:F}
 F \colon \Mbartri \dasharrow \Mbarrel.
\end{equation}

To describe
$
F
$
in coordinates,
take a general point
$
 p = (u:v:w) \in \bP^2.
$
Since
$
S(3)
$
is the blowup of
$
\bP^2
$
in the base locus of the Hesse pencil,
we can regard
$
p
$
as a point of
$
S'(3) \setminus (D_0 \cup \cdots \cup D_8)
$.
The elliptic triple $(E, L_0, L_1)$
associated with $p$
is given by 
\begin{align}
\label{eq:elliptic_triple_in_coordinates}
\begin{cases}
 E =
\{
 (x:y:z) \in \bP^2 \mid u v w ( x^3 + y^3 + z^3 ) - ( u^3 + v^3 + w^3) x y z = 0
\},
\\
L_0
=\scO_E(1)
=\scO_E(3o),
\\
L_1
= \scO_E(3p).
\end{cases}
\end{align}

\begin{lemma} \label{lm:associated_quadruple_in_coordinates}
The quadruple associated to the triple \pref{eq:elliptic_triple_in_coordinates}
is
\begin{equation}\label{eq:associated_quadruple_in_coordinates}
( \bk^3_{x_0  y_0 z_0} , \bk^3_{x_1 y_1 z_1} , \bk^3_{x_2 y_2 z_2}
, \bk N_{uvw} ),
\end{equation}
where
$
N_{uvw}
\in
\bk^3_{x_0  y_0  z_0} \otimes \bk^3_{x_1 y_1 z_1}
\otimes
\bk^3_{x_2 y_2 z_2}
$
is given by
\begin{equation}\label{eq:Vinberg_normal_form}
w ( x_0 x_1 x_2 + y_0 y_1 y_2 + z_0 z_1 z_2)
+ u ( x_0 z_1 y_2 + y_0 x_1 z_2 + z_0 y_1 x_2)
+ v ( x_0 y_1 z_2 + y_0 z_1 x_2 + z_0 x_1 y_2).
\end{equation}
\end{lemma}

\begin{proof}
It suffices to show that the triple associated to
\pref{eq:associated_quadruple_in_coordinates}
coincides with
\pref{eq:elliptic_triple_in_coordinates}.
We follow \cite[Theorem 1]{MR1361786}.

If we identify
\pref{eq:Vinberg_normal_form}
with a
$
3 \times 3
$
matrix of linear functions in
$
(x_0 , y_0, z_0),
$
it is given by
\begin{equation}\label{eq:matrix_in_linear_funcitons_of_x}
M(x)
=
\begin{pmatrix}
w x_0 & v z_0 & u y_0
\\
u z_0 & w y_0 & v x_0
\\
v y_0 & u x_0 & w z_0
\end{pmatrix}.
\end{equation}
Hence, by \cite[Proposition 1 (iii)]{MR1361786},
the elliptic curve associated to
\pref{eq:associated_quadruple_in_coordinates}
is
\begin{equation}
\det M(x) = 0
\iff
(u^3 + v^3 + w^3) x_0 y_0 z_0
-
uvw (x_0^3 + y_0^3 + z_0^3 )
=
0.
\end{equation}
Thus we see the coincidence of the elliptic curves,
which will be denoted by $E$.

Next we consider the complete intersection
\begin{equation}\label{eq:graph_of_the_translation}
E' \hookrightarrow
\bP^2_{x_0y_0z_0} \times \bP^2_{x_1y_1z_1}
\end{equation}
defined by the derivatives
\begin{equation}\label{eq:derivatives}
\begin{split}
\partial_{x_2}N_{uvw}
&=
v y_0 z_1 + u z_0 y_1 + w x_0 x_1,
\\
\partial_{y_2}N_{uvw}
&=
v z_0 x_1 + u x_0 z_1 + w y_0 y_1,
\\
\partial_{z_2}N_{uvw}
&=
v x_0 y_1 + u y_0 x_1 + w z_0 z_1
\end{split}
\end{equation}
in
$
\bk^3_{x_0y_0z_0} \otimes \bk^3_{x_1y_1z_1}.
$
Again by
\cite[Proposition 1 (iii)]{MR1361786},
$
E'
$
is canonically identified with
$
E
$
by the first projection.

It is clear from the definition that
\begin{equation}
\cO_{\bP^2 \times \bP^2}(1,0)|_E
\cong
\cO_{\bP^2_{x_0y_0z_0}}(1)|_E
\cong
\cO_E(3o).
\end{equation}

On the other hand, as explained in
\pref{sc:AS-regular},
the embedding
\pref{eq:graph_of_the_translation}
coincides with the graph of the translation by the point
$
( v : u : w )
\in
E
$,
with respect to the origin
$
o
=
( 1 : -1 : 0 )
$.
To see this, compare
\pref{eq:graph}
with
\pref{eq:derivatives}.
Hence we see
\begin{equation}
\cO_{\bP^2 \times \bP^2} (0 , 1)|_E
\cong
\tau_{(v:u:w)}^*\cO_E(3o)
\cong
\cO(3(\ominus (v : u : w))),
\end{equation}
where
$
\ominus
$
denotes the inversion by the group low
of
$
E
$
(with respect to the origin
$
o
$
).
Since we know
$
\ominus (a : b: c )
=
(b:a:c)
$
in this case, we see
$
\cO_{\bP^2 \times \bP^2} (0 , 1)|_E
\cong
\cO_E(3 p).
$
This concludes the proof of
\pref{lm:associated_quadruple_in_coordinates}.
\end{proof}

\begin{comment}
Therefore we see that the rational map
$
F
$
is described as
\begin{equation}
(w:u:v) \mapsto N_{uvw}
\end{equation}
in coordinates.
\end{comment}

Let
$
G_{648}
$
be the inverse image of
$
G_{216}
$
under the natural morphism
$
\SL_3(\bk) \to \PGL_3(\bk)
$.

\begin{lemma}\label{lm:invariant_theory_of_the_Weyl_group_by_Vinberg}
The map
\begin{align} \label{eq:N}
N  \colon
\bP^2_{u:v:w}
\to
\bP(\bk^3 \otimes \bk^3 \otimes \bk^3) ;
\quad
( u : v : w ) \mapsto N_{uvw}
\end{align}
descends to an isomorphism
\begin{align} \label{eq:Nbar}
\Nbar
\colon
\bP^2/G_{648}
 \simto
\Mbarrel.
\end{align}
\end{lemma}

\begin{proof}
Let $m$ be a positive integer and
consider a $\bZ/m$-graded Lie algebra
$
 \frakg = \bigoplus_{i \in \bZ/m} \frak{g}_i.
$
Let further $G$ be a linear algebraic group
whose Lie algebra is isomorphic to $\frakg$, and
$
G_0 \subset G
$
be the connected closed subgroup corresponding to
$
\frakg_0
$.
Note that the group
$
G_0
$
naturally acts on
$
V := \frakg_1
$
by conjugation.
The universal cover of $G_0$
will be denoted by $\Ghat_0$.
Vinberg \cite{MR0430168}
introduced the notions of the Cartan subspace
$
 \frakc \subset V,
$
the Weyl group $W$,
and the action of $W$ on $\frakc$.
He showed that the natural morphism
\begin{equation}\label{eq:invariant_theory_of_weyl_group}
\bk[V]^{G_0}
\to
\bk[\frakc]^{W}
\end{equation}
is an isomorphism
under some assumption
\cite[Theorem 7]{MR0430168}.
We apply his results to the Lie algebra
$
\frakg
$
which is named as No.2 in
\cite[p. 491, Table]{MR0430168}.
According to
\cite[p.\ 409, $2^{0}$]{MR0486173},
we obtain the standard action of the group
$
\Ghat_0
=
SL(3) \times SL(3) \times SL(3)
$
on
$
V
=
\frakg_1
=
\bk^3 \otimes \bk^3
\otimes \bk^3,
$
and the Cartan subspace
$
\frakc
\subset
V
$
is spanned by the three elements
\begin{equation}
\begin{cases}
e_1 \otimes e_3 \otimes e_2
+
e_2 \otimes e_1 \otimes e_3
+
e_3 \otimes e_2 \otimes e_1,
\\
e_1 \otimes e_2 \otimes e_3
+
e_2 \otimes e_3 \otimes e_1
+
e_3 \otimes e_1 \otimes e_2,
\\
e_1 \otimes e_1 \otimes e_1
+
e_2 \otimes e_2 \otimes e_2
+
e_3 \otimes e_3 \otimes e_3,
\end{cases}
\end{equation}
where
$
\langle
e_1, e_2 , e_3
\rangle
$
is the standard basis of
$
\bk^3
$.
The Weyl group
$
W
$
is isomorphic to the simple complex reflection group No. 25 of
\cite[p.301, table VII]{MR0059914},
which is nothing but
$
G_{648}
$.
Hence we obtain an isomorphism
\begin{equation}
\bk[V]^{\Ghat_0}
 \simto
\bk[\frakc]^{G_{648}},
\end{equation}
as a special case of
\pref{eq:invariant_theory_of_weyl_group}.
In coordinates, this isomorphism coincides with the map \eqref{eq:N}.
\end{proof}

\begin{corollary}
The rational map $F$ in \eqref{eq:F}
is the composition of the contraction
$
 \Mbartri \to \bP^2/G_{648}
$
of the smooth rational curve obtained as the image of the sections
$
D_0 , \dots , D_8
$,
and the isomorphism $\Nbar$
in \eqref{eq:Nbar}.
\end{corollary}

%
%

\subsection{An isomorphism between $\Mbartri$ and $\Mbarst$}
 \label{sc:isom_with_Mbar_1,2}

We show the existence of a natural isomorphism
$
\Mbartri
 \simto
\Mbarst,
$
where
$
\Mbarst
$
is the coarse moduli scheme of
stable two-pointed genus one curve.
Through this isomorphism,
we can identify the geometry of
$
\Mbartri
$
with known results on
$
\Mbarst
$.

\begin{proposition}\label{pr:isom_with_Mbar_1,2}
There exists a natural isomorphism
\begin{equation}\label{eq:isom_with_Mbar_1,2}
\psi
\colon
\Mbartri
=
S(3) / G_{216}
\simto
\Mbarst,
\end{equation}
which is compatible with the standard isomorphism
\begin{equation}
X(3) / \PSL ( 2 , \bF_3)
=
X(1)
\cong
\Mbarstb.
\end{equation}
\end{proposition}
\begin{proof}
Since
$
G_{216}
\cong
( \bZ / 3 )^2 \rtimes \SL ( 2 , \bF_3 )
$,
we can regard
$
\Mbartri
$
as the quotient of
$
\Sbar(3)
=
S(3) / ( \bZ / 3 )^2
$
by
$
\SL ( 2 , \bF_3 )
$.
As explained in
\cite[Proposition 5.1]{Artebani-Dolgachev},
the surface
$
\Sbar(3)
$
has four
$
A_2
$-singularities,
and its crepant resolution is isomorphic,
as a fibration over
$
X(3)
$,
to
$
S(3)
$.
The fibration
$
\pibar \colon \Sbar(3) \to X(3)
$
has four singular fibers of type
$
\mathrm{I}_1
$,
and hence we can regard
$
\pibar
$
as a family of stable one-pointed genus one curves;
in fact, the section
$
E
\subset
\Sbar(3)
$
can be regarded as the universal marked point.
Thus we obtain the classifying morphism
$
c
\colon
X(3) \to \scMbar_{1,1}
$
and the commutative diagram
\begin{align*}
\xymatrix{
\ar @{} [dr] |{\square}
\Sbar(3) \ar[r]^{24}_{\Psi} \ar[d]^{\pibar}
&
\ar @{} [dr] |{\circlearrowleft}
\scMbar_{1,2} \ar[r]^{1} \ar[d]^{u}
&
\Mbarst \ar[d]
\\
X(3) \ar[r]^{24}_{c}
&
\scMbar_{1,1} \ar[r]^{\frac{1}{2}}
&
\Mbarstb.}
\end{align*}
In the diagram above,
the numbers on the arrows indicate their degrees.
The morphism
$
u
$
is the universal family.
We can also understand
$
\Psi
$
as the classifying morphism;
given
$
p
\in
\Sbar(3)
$,
we can define the associated stable two-pointed genus one curve as follows.
The underlying curve is the fiber
$
E_p
:=
\pibar^{-1}(\pibar(p))
$,
the first marked point is the intersection with
the section
$
E
$,
and the second marked point is
$
p
$.
If
$
p
$
is the singular point of
$
E_p
$,
we consider the blow-up of
$
\Sbar(3)
$
in
$
p
$
and use the inverse image of
$
E_p
$,
which is a curve of type
$
\mathrm{I}_2
$,
as the underlying curve.

Now we can check that the action of
$
\SL(2,\bF_3)
$
on
$
\Sbar(3)
$
respects the finite surjective morphism
$
\Sbar(3)
\to
\Mbarst
$, so that we obtain a birational morphism
$
\psi
\colon
\Mbartri
=
\Sbar(3) / \SL(2, \bF_3)
\twoheadrightarrow
\Mbarst
$
over
$
X(3) / \PSL(2 , \bF_3)
=
X(1)
=
\Mbarstb
$.
Since
$
\psi
$
is a finite birational morphism between
normal projective varieties,
it is an isomorphism.
\end{proof}

The boundary
$
\Mbarst
\setminus
M_{1,2}
$
is the union of two smooth rational curves
$
\Delta_{\irr}
$
and
$
\Delta_{0,2}
$.
The former is the fiber over the cusp of
$
X(1)
\cong
\Mbarstb
$, and the latter is the section of
$
\Mbarst
\to
\Mbarstb
$.
Since the morphism
$
\psi
$
comes from the classifying morphism
$
\Psi
$,
we see that under the isomorphism
$
\psi
$,
the image of the four singular fibers of
$
\pibar
$
in
$
\Mbartri
$
is identified with
$
\Delta_{\irr}
$.
Similarly, the section
$
\Ebar := \Psi(E)
$
is identified with
$
\Delta_{0,2}
$.

Birational geometry of
$
\scMbar_{1,2}
$
and its coarse moduli scheme are studied in
\cite{MR2801404,MR2862065} and
\cite{MR3174737}.
In particular, the following facts are proved:
\begin{itemize}
 \item
$
\scMbar_{1,2}
$
is isomorphic to the global quotient of the total space
of the Weierstrass family.
 \item
 There exists a birational morphism
 \begin{equation}\label{eq:contraction_of_M1,2bar}
 G \colon \Mbarst \to \Mbarst(1)
 \end{equation}
which contracts the divisor
$
\Delta_{0,2}
\subset
\Mbarst
$
to a point. Here
$
\Mbarst(1)
$
is the coarse moduli space of
$
1
$-stable
$
2
$-pointed genus
$
1
$
curves, and the point
$
G(\Delta_{0,2})
$
represents the cuspidal plane cubic curve
with two marked points
(see \cite[p. 1844]{MR2862065} and
\cite[Lemma 4.1]{MR2862065}).

 \item
There exists an explicit isomorphism
$
\Mbarst(1)
\simto
\bP(1,2,3)
$
through which the morphism
$
G
$
of
\pref{eq:contraction_of_M1,2bar}
is identified with a weighted blow-up of
$
\bP(1,2,3)
$
at the point
$
(1:0:0)
$
(see
\cite[Section 2]{MR3174737}).

 \item
The image
$
G ( \Delta_{\irr} )
\subset
\Mbarst(1)
$
is identified with the cuspidal curve given by
\begin{equation}\label{eq:Weierstrass}
 \{ (x:y:z) \in \bP(1,2,3) \mid 4y^3-27z^2 = 0 \}.
\end{equation}
The defining equation comes from the discriminant of the Weierstrass family
(see \cite[Theorem 2.3]{MR3174737}).
\end{itemize}

Since the morphism
$
F
$
contracts the section
$
\Ebar
\subset
\Mbartri
$
and
$
\psi
$
identifies it with
$
\Delta_{0,2}
$,
we obtain an isomorphism
\begin{equation}
\sigma \colon \bP(6,9,12) \simto \bP(1,2,3)
\end{equation}
which fits into the following commutative diagram:
\begin{align*}
\xymatrix{
\ar @{} [dr] |{\circlearrowleft}
\Mbartri \ar[r]^{\psi} \ar[d]_{F}
&
\Mbarst \ar[d]^{G}
\\
( \Mbarrel \cong ) \bP(6,9,12)
\ar[r]_(0.6){\sigma}
&
\bP(1,2,3)}
\end{align*}
Finally, the boundary of
$
\Mbarrel
$
is the prime divisor
\begin{align}
 \Deltarel
  := \Mbarrel \setminus \Mrel
  =  F ( \Ebar ).
\end{align}
The morphism
$
\sigma
$
identifies
$
\Deltarel
\subset
\Mbarrel
$
with the cuspidal curve
\pref{eq:Weierstrass} and
the point
$
F(\Delta_{0,2})
\in
\Deltarel
$
is mapped to the cusp
$
(1:0:0)
$.
In particular,
the restriction of $F$ to
$
\Delta_{\irr} \cong \Ebar
\subset
\Mbartri
$
gives the normalization of
$
\Deltarel.
$

\subsection{Interpretation of points on the boundary}

Following \cite{MR1348793},
we look more closely at the boundary of
$
\Mbartri
\cong
\Mbarst
$
and
$
\Mbarrel
$.
The poset of semi-stable orbits,
whose partial order is defined by inclusions
of the closures of the orbits,
is described in
\cite[p. 92]{MR1348793}.
\begin{itemize}
 \item
Any closed point of the divisor
$
\Delta_{0,2}
$
gives rise to a triple
$
( E , L_0 , L_1 )
$
satisfying
$
L_0
\cong
L_1
$,
which in turn corresponds to the commutative
$
\bP^2
$.
All those points are mapped to the cusp of
$
\Deltarel
$
under the contraction
$
F=G
$.
It is called the Veronese point and representing closed orbit $O(\Ver)$
corresponds to the linear triple,
which is associated with the commutative $\bP^2$.

\item
The image of the four singular points of
$\Sbar(3)$ in $\Mbartri$
is represented by the closed orbit
$O(\I_6)$.
It corresponds to the triple
$
(E , L_0 , L_1),
$
where $E$ is the singular genus one curve of type $\I_6$,
and the line bundles $L_0$ and $L_1$ have multi-degrees
$
 (1 , 0 , 1 , 0 , 1 , 0)
$
and
$
 (0 , 1 , 0 , 1 , 0 , 1)
$
respectively.
This is not an admissible triple
in the sense of \pref{df:admissible}.
The associated graded algebra is isomorphic to
$
 \bk \langle x , y , z \rangle / (x^2 , y^2 , z^2),
$
which is not an AS-regular algebra.
 \item
The closed orbits
corresponding to the rest of $\Deltarel$ are
$
 O(\I_3^{\lambda^3}(a)).
$
They are in one-to-one correspondence
with isomorphism classes of admissible triples
$
(E , L_0 , L_1)
$
such that both $L_0$ and $L_1$ embed $E$ into $\bP^2$
as a triangle of three lines.
\end{itemize}

%
%

\appendix

\section{Proof of \pref{th:SVdB}}
\label{sc:VdB}

In this section,
we give a proof of \pref{th:SVdB},
which we learned from Van den Bergh.
\pref{th:SVdB} is stated in \cite{MR1816070},
although the proof cannot be found in the literature
to the authors' best knowledge.

Let $S$ be a quadratic AS-regular algebra of dimension 3.
The Hilbert series of a finitely-generated graded $S$-module $M$
is defined by
\begin{align}
 h_M(t)
  =\sum_{i\in\bZ}(\dim_{\bk}M_i)t^i
  \in \bZ (\!( t )\!).
\end{align}
It follows from \eqref{eq:AS-regular} that
the Hilbert series of $S$ is given by
\begin{align}
 h_S(t)=(1-t)^{-3}.
\end{align}
Since $M$ admits a finite projective resolution,
there exists a Laurent polynomial $q_M(t)$ such that
\begin{align}
 h_M(t) = q_M(t) h_S(t).
\end{align}
The Laurent polynomial $q_M(t)$ is called
the \emph{characteristic polynomial} of $M$.
We expand $q_M(t) \in \bZ \ld t^{\pm 1}\rd$
around $t=1$ as
\begin{align}
 q_M(t)=r+a(1-t)+b(1-t)^2+f(t)(1-t)^3,
\end{align}
where $r,a,b\in\bZ$ and $f(t)\in\bZ \ld t^{\pm 1} \rd$.
The integer $r$ is called the \emph{rank} of the module $M$.
The \emph{Gelfand-Kirillov dimension}
(or the \emph{GK-dimension} for short)
is defined as the growth rate of $\dim_{\bk} M_i$
as a function of $i$;
\begin{align}
\GKdim(M) =
\begin{cases}
 3 & r > 0, \\
 2 & r=0 \text{ and } a>0, \\
 1 & r=a=0 \text{ and } b>0, \\
 0 & r=a=b=0 \text{ and } f(1) = \dim_\bk M > 0.
\end{cases}
\end{align}
%
%
Let $\pi \colon \gr S \to \qgr S$ be the natural quotient functor,
which is exact.
The GK-dimension of an object of $\qgr S$ is defined by
\begin{align}
 \GKdim(\pi(M))=\GKdim(M)-1. 
\end{align}
An object $\cM\in\qgr S$ is said to be \emph{torsion-free}
if it is isomorphic to an object of the form $\pi(M)$ such that
any non-trivial submodule of $M$ has GK-dimension three.

\begin{lemma}[{\cite[Lemma 2.2.1]{MR2708379}}]
An object $0\not=\scM\in\qgr{S}$ is torsion-free if and only if
$\Hom(\scN,\scM)=0$ holds for any $\scN\in\qgr{S}$ with $\GKdim(\scN)\le 1$.
\end{lemma}

The objects $\{\scO(m)\}_{m\in\bZ}$ have the following characterization.

\begin{proposition} \label{pr:characterization1}
The followings are equivalent for an object $\scI$ of the non-commutative projective plane $\qgr{S}$.
 \begin{enumerate}[(1)]
 \item There exists an integer $m\in\bZ$ such that $\scI$ is isomorphic to $\scO(m)$.
 \item $\scI$ is torsion-free of rank one and satisfies the condition
 \begin{align}
 \chi(\scI , \scI )=\sum_i(-1)^i\dim\Ext^i( \scI , \scI )=1.
 \end{align}
 \end{enumerate}
 \end{proposition}
 
\begin{proof}
Recall from \cite[Section 2.2.2]{MR2708379}
that the Grothendieck group $K(\qgr{S})$ is freely generated
by the classes
$
[\cO], [\cS], [\cP]
$
of the free module $\cO$,
a line module $\cS$,
and a point module $\cP$.
Let $\scI$ be a torsion-free object of rank one.
One can assume that the class $[\cI]$ in the Grothendieck group
$K(\qgr{S})$ is equal to $[\scO]-n[\scP]$ for some $n \in \bZ$
by replacing $\cI$ with $\cI(m)$ for some $m \in \bZ$ if necessary
\cite[Section 2.2.3]{MR2708379}.
Such an object is isomorphic to $\cO$
if and only if $n=0$
\cite[Theorem 2.2.11]{MR2708379}.
Now \pref{pr:characterization1} follows from
$ 
 \chi(\cI(m), \cI(m)) = \chi( \scI , \scI)=1-2n.
$
\end{proof}

Next we rephrase the notions appearing in
\pref{pr:characterization1}
in terms of the intrinsic structures of the abelian category $\qgr{S}$.

\begin{definition}
Let $\scC=\qgr{S}$ be a non-commutative projective plane.
Denote by $\scC_0$ the Serre subcategory
of $\scC$ consisting of objects with finite length.
We inductively define the categories $\scC_i$ ($i>0$) to be
the Serre subcategories of the objects $\scM\in\scC$ whose
images in the quotient category $\scC/\scC_{i-1}$ have finite lengths.
Let $\scCtilde_i \subset \grmod(S)$ be the
pre-image of $\scC_i$ under the functor $\pi$, and
set $\scCtilde_{-1}=\tors(S)$. 
\end{definition}

%
We learned the following definition from Izuru Mori.

\begin{definition}
The Krull dimension $\Kdim \cM$ of an object $\cM \in \cC$
is the integer $i$ such that
$\cM \in \cC_i$ and $\cM \nin \cC_{i-1}$.
\end{definition}

\pref{pr:Kdim=GKdim} below is a special case
of a more general result
in \cite{Mori_in_preparation}.
We reproduce his proof below
for the sake of completeness.
Mori also pointed out that \pref{pr:Kdim=GKdim} also follows
from \cite[Proposition 2.4 and Theorem 3.1]{MR1438182}.

\begin{proposition} \label{pr:Kdim=GKdim}
For any object $\cM \in \cC$,
one has
\begin{align}
 \Kdim \cM = \GKdim \cM.
\end{align}
\end{proposition}

\begin{proof}
By definitions, it suffices to show
\begin{equation}
\Kdim M = \GKdim M
\end{equation}
for any object
$
M \in \grmod S
$.
It is known (see \cite{Artin_Tate_Van_den_Bergh})
that there exists a central element
$
g
\in
S_3
$
of degree three such that
$
S/(g)
$
is isomorphic to the twisted homogeneous coordinate ring of the triple
$
(E , L_0 , \sigma )
$.
In particular, the assertion is true for any object of the category
$
\grmod S/(g)
$. This immediately implies that the assertion is true for any object
$
M
\in
\grmod S
$
satisfying
$
M \cdot g=0
$.

If
$
M \cdot g^{\ell} = 0
$
for some
$
\ell > 0
$,
then by using the filtration
\begin{equation}
0 = M \cdot g^{\ell}
\subset
M \cdot g^{\ell -1}
\subset
\cdots
\subset
M \cdot g
\subset
M
\end{equation}
whose subquotients are annihilated by
$
g
$,
we can again check the assertion for such
$
M
$
and see
$
M
\in
\scCtilde_1
$.

Next, take
$
M
\in
\grmod S
$
such that the homomorphism
\begin{equation}
\cdot g \colon M \to M(3)
\end{equation}
is injective.
Suppose for a contradiction that
$
\Kdim M = 2
$.
Since
$
\pi
\colon
\scCtilde_1
\to
\scCtilde_1 / \scCtilde_{0}
$
is exact (\cite[Lemma A.2.3]{MR1812507}),
we obtain an exact sequence
\begin{equation}
0 \to \pi(M) \xto[]{\cdot g}  \pi(M) \to \pi (M/(g)) \to 0.
\end{equation}
Since
$
\pi (M/(g)) \not= 0
$,
this contradicts the fact that objects of the category
$
\scCtilde_{1}/ \scCtilde_{0}
$
have finite length.
Therefore we see
$
\Kdim M = 3
$,
and also see that
$
\GKdim M
=
\GKdim M/(g) +1
=
3
$.

Finally let
$
M
\in
\grmod S
$
be an arbitrary object.
By setting
$
\tau M
=
\ker ( M \xto[]{\cdot g } M)
$,
we can combine our results obtained so far to see
\begin{equation}
\Kdim M
=
\max \{ \Kdim M/\tau M , \Kdim M \}
=
\max \{ \GKdim M/\tau M , \GKdim M \}
=
\GKdim M.
\end{equation}
Thus we conclude the proof.
\end{proof}

\begin{proposition} \label{pr:rank=length}
The rank of an object of $\cC$ coincides
with its length in $\cC/\cC_1$.
\end{proposition}

\begin{proof}
We prove \pref{pr:rank=length}
by induction on the rank of an object $M \in \cC$.
The case when $\rank{M}=0$ is an immediate consequence of \pref{pr:Kdim=GKdim}.
Let $r$ be a positive integer, and
assume that \pref{pr:rank=length}
for all $N \in \cC$ with $\rank{N}<r$.
Let $M \in \cC$ be an object of rank $r$, and
$T \subset M$ be the maximal submodule of GK-dimension less than three,
whose existence is guaranteed since $M$ is Noetherian.
Since we are interested in the length of $M$
in $\scCtilde/\scCtilde_1$, we can replace $M$ with $M/T$, so that
$M$ has no non-trivial submodule of GK-dimension less than three
(in the terminology of \cite{MR2708379}, we assumed that $M$ is \emph{pure}).

Since $M\not=0$,
for sufficiently large $m>0$,
there exists a non-trivial morphism
\begin{align}
 f\colon S\to M(m).
\end{align}
By the purity of $M$, we see that $\Image(f)\subset M(m)$
has GK-dimension three. Moreover the existence of a surjective morphism $S\twoheadrightarrow\Image(f)$
shows $\rank\Image(f)=1$. By an easy calculation of the Hilbert polynomials, we see that
the rank of $\coker(f)$ is $r-1$. By the induction hypothesis, this means that
the length of $\coker(f)$ is $r-1$.
Since the morphism $M\twoheadrightarrow \coker(f)$ is non-trivial in $\scCtilde/\scCtilde_1$,
we see that the length of $M$ is at least $1+(r-1)=r$.
Conversely, by looking at the Hilbert polynomial of $M$, it is easy to show that
the length of $M$ in $\scCtilde/\scCtilde_1$ is at most $r$.
This concludes the proof of \pref{pr:rank=length}.
\end{proof}

Propositions \ref{pr:characterization1},
\ref{pr:Kdim=GKdim}, and \ref{pr:rank=length}
immediately implies the following:

\begin{corollary}\label{cr:Intrinsic_characterization_of_the_structure_sheaf_modulo_shifts}
An object $\scI\in\scC$ is isomorphic to $\scO(m)$ for some $m\in\bZ$
if and only if it satisfies the following properties:
\begin{itemize}
\item $\Hom(\scN, \scI)=0$ for all $\scN\in\scC_1$.
\item $\scI$ has length one in $\scC/\scC_1$.
\item $\chi(\scI,\scI)=1$.
\end{itemize}
\end{corollary}

Since the subcategories $\cC_i$
and the lengths of objects
are intrinsic notions of the category $\cC$,
one can identify the set of objects $\{ \cO(m) \}_{m\in\bZ}$
up to isomorphisms
using only the structure of $\cC$ as an abelian category.
This concludes the proof of the following:

\begin{theorem}[{\cite[Section 11.2]{MR1816070}}] \label{th:VdB}
Let $S$ be a $3$-dimensional quadratic AS-regular algebra.
Then the sequence
$
 ( \scO(n) )_{n\in\bZ}
$
of objects is determined
up to shift
by the structure of $\qgr S$
as an abelian category.
\end{theorem}

There exists a natural functor
\begin{align}
 (\check{-})\colon \GrAlg(\bZ)\to \Alg(\bZ) 
\end{align}
from the category of $\bZ$-graded algebras
to the category of $\bZ$-algebras,
which associates the $\bZ$-algebra
$\Av = \bigoplus_{i,j\in \bZ} \Av_{ij}$
to a $\bZ$-graded algebra $A = \bigoplus_{i \in \bZ} A_i$
where $\Av_{ij}=A_{j-i}$.
A pair of $\bZ$-graded algebras $A$ and $B$ satisfy $\Av\cong\Bv$
if and only if $A$ is isomorphic
to the Zhang twist of $B$
\cite[Theorem 1.2]{MR2776790}.
This is also equivalent to the existence
of an equivalence $\Phi\colon \gr A \to \gr B$
sending $A(n)$ to $B(n)$ for all $n\in \bZ$
\cite[Theorem 1.1]{MR2776790}.
A $\bZ$-algebra $A$ is \emph{$m$-periodic}
if $A$ is isomorphic to the shifted $\bZ$-algebra $A(m)$
defined by
$
 A(m)_{ij}=A_{i+m,j+m}.
$
A $\bZ$-algebra is isomorphic to $\Bv$
for some $\bZ$-graded algebra $B$
if and only if it is $1$-periodic
\cite[Lemma 2.4]{MR2836401}.


\begin{theorem}[{\cite[Theorem 3.4]{MR2836401}}]
Any quadratic AS-regular $\bZ$-algebra of dimension 3 is one-periodic.
\end{theorem}

Hence any noncommutative projective plane
is equivalent to $\qgr A$
for some quadratic AS-regular $\bZ$-algebra $A$ of dimension 3.
For a quadratic AS-regular algebra $S=S(E,\sigma,L_0)$
associated with a triple $(E,\sigma, L_0)$,
the $\bZ$-algebra $A$ associated with the triple
$(E,L_0,L_1:=\sigma^* L_0)$ satisfies $A=\Sv$.


\begin{proof}[Proof of \pref{th:SVdB}]
We show that
for any $3$-dimensional quadratic regular $\bZ$-algebra $A$,
one can recover its isomorphism class
from the $\bk$-linear abelian category $\qgr A$.

Let $(C, L_0, L_1)$ be a triple
which gives rise to $A$, and
$\sigma$ be an automorphism of $C$ such that
$\sigma^*L_0\cong L_1$.
Recall that $A$ is isomorphic to $\Sv$,
where $S$ is the quadratic regular algebra
associated to $(C, \sigma, L_0)$.
In particular the category $\qgr A$ is 
equivalent to $\qgr S $.

By \pref{cr:Intrinsic_characterization_of_the_structure_sheaf_modulo_shifts},
one can identify the set of objects
$
 \{\scO(n)\}_{n\in\bZ}\subset\qgr{S}
$
up to isomorphisms
from the abelian category $\qgr{S}$.
Since $\Hom(\scO(m),\scO(n))=0$ if and only if $m>n$,
one can also recover the order among the objects $\{\scO(n)\}_{n\in\bZ}$
from the structure of $\qgr S$ as a category.
Hence one can recover the $\bZ$-algebra
\begin{align}
 \bigoplus_{m,n\in\bZ}\Hom(\scO(m),\scO(n))
\end{align}
up to shifts. Since the $\bZ$-algebra $A$ is associated to a graded algebra,
it is one periodic. Hence the ambiguity by shifts does not affect the isomorphism class of the
$\bZ$-algebra, and the $\bZ$-algebra thus obtained is isomorphic to $A$.
\end{proof}

\begin{corollary} \label{cr:Mori_question}
For 3-dimensional quadratic AS-regular algebras $S$ and $S'$,
the following are equivalent:
\begin{enumerate}[(1)]
 \item \label{it:associated_Z_algebras_are_isomorphic}
$\Sv\cong\Sv'$.
 \item \label{it:isomorphic_to_a_Zhang_twist}
$S$ is isomorphic to a Zhang twist of $S'$.
 \item \label{it:grmod_are_equivalent}
$\grmod(S)\cong\grmod(S')$.
 \item \label{it:qgr_are_equivalent}
$\qgr{S}\cong\qgr(S')$.
\end{enumerate}
\end{corollary}

\begin{proof}
(\pref{it:associated_Z_algebras_are_isomorphic})
is equivalent to
(\pref{it:isomorphic_to_a_Zhang_twist})
 by
\cite[Theorem 1.2]{MR2776790},
and
(\pref{it:isomorphic_to_a_Zhang_twist})
is equivalent to
(\ref{it:grmod_are_equivalent})
by
\cite[Theorem 3.5]{MR1367080}.
(\ref{it:grmod_are_equivalent})
clearly implies
(\ref{it:qgr_are_equivalent}),
and 
(\ref{it:qgr_are_equivalent})
implies
(\ref{it:associated_Z_algebras_are_isomorphic})
by \pref{th:SVdB}.
\end{proof}

This answers a question of Mori \cite{MR2234308}
for non-commutative projective planes.




\bibliographystyle{amsalpha}
\bibliography{bibs}

\noindent
Tarig Abdelgadir

Mathematics Section,
The Abdus Salam
International Centre for Theoretical Physics,
Strada Costiera 11,
I - 34151,
Trieste,
Italy.

{\em e-mail address}\ : \  tarig.m.h.abdelgadir@gmail.com

\ \vspace{0mm} \\

\noindent
Shinnosuke Okawa

Department of Mathematics,
Graduate School of Science,
Osaka University,
Machikaneyama 1-1,
Toyonaka,
Osaka,
560-0043,
Japan.

{\em e-mail address}\ : \  okawa@math.sci.osaka-u.ac.jp
\ \vspace{0mm} \\

\ \vspace{0mm} \\

\noindent
Kazushi Ueda

Department of Mathematics,
Graduate School of Science,
Osaka University,
Machikaneyama 1-1,
Toyonaka,
Osaka,
560-0043,
Japan.

{\em e-mail address}\ : \  kazushi@math.sci.osaka-u.ac.jp
\ \vspace{0mm} \\

\end{document}